\documentclass[11pt]{article}
\usepackage[right=2.5cm,left=2.5cm,top=2cm,bottom=3cm,headsep=2cm,footskip=2cm]{geometry}
\usepackage{graphicx}  
\usepackage{enumerate}
\usepackage{amsmath}
\usepackage{amssymb}
\usepackage{color}
\usepackage{float}
\usepackage{cite}
\usepackage{mathtools}
\usepackage{hyperref}
\hypersetup{
	colorlinks=true,        
	linkcolor=black,        
	citecolor=blue,         
	urlcolor=blue           
}

\usepackage{authblk}

\title{\bf Tracking and regret bounds for online zeroth-order Euclidean and Riemannian optimisation\footnote{This work is supported by the Australian Research
		Council (DP210102454) and the Australian Government, via grant
		AUSMURIB000001 associated with ONR MURI grant N00014-19-1-2571.}}
\author[1]{Alejandro I. Maass}
\author[1]{Chris Manzie}
\author[1]{Dragan Ne{\v{s}}i{\'c}}
\author[1]{Jonathan H. Manton}
\author[2]{Iman Shames}
\affil[1]{Department of Electrical and Electronic Engineering, The University of Melbourne}
\affil[2]{CIICADA Lab, College of Engineering \& Computer Science, The Australian National University}
\affil[*]{e-mails: \small \texttt{\{maassa,manziec,dnesic,jmanton\}@unimelb.edu.au;iman.shames@anu.edu.au}  }

\date{}

\newcommand{\espu}[1]{\mathbb{E}_{u_0}\left[#1\right]}
\newcommand{\espp}[1]{\mathbb{E}\left[#1\right]}
\newcommand{\espc}[2]{\mathbb{E}\left[\left. #1\, \right| \, #2\right]}
\newcommand{\qed}{\hfill \ensuremath{\blacksquare}}
\newcommand{\R}{\mathbb{R}}
\newcommand{\N}{\mathbb{N}}

\newcommand{\inner}[1]{\left\langle #1 \right\rangle}

\newcommand{\norm}[1]{\left\|#1\right\|}
\newcommand{\trace}[1]{\mbox{Tr}\left\{#1\right\}}
\newcommand{\non}{\nonumber}
\newcommand{\X}{\mathcal{X}}
\newcommand{\M}{\mathcal{M}}
\newcommand{\grad}{\mbox{\rm grad}}
\newcommand{\reg}{\mbox{\rm Reg}}
\newcommand{\Exp}{\mathrm{Exp}}
\newcommand{\ie}{i.e.~}
\newcommand{\eg}{e.g.~}

\newtheorem{assumption}{Assumption}
\newtheorem{proposition}{Proposition}
\newtheorem{definition}{Definition}
\newtheorem{theorem}{Theorem}
\newtheorem{corollary}{Corollary}
\newtheorem{lemma}{Lemma}
\newtheorem{remark}{Remark}

\providecommand{\proofname}{Proof}
\newenvironment{proof}%
{%
	\par\noindent{\bfseries\upshape \proofname\ }%
}%
{\qed}

\begin{document}
\maketitle

\begin{abstract}
We study numerical optimisation algorithms that use zeroth-order information to minimise time-varying geodesically-convex cost functions on Riemannian manifolds. In the Euclidean setting, zeroth-order algorithms have received a lot of attention in both the time-varying and time-invariant cases. However, the extension to Riemannian manifolds is much less developed. We focus on Hadamard manifolds, which are a special class of Riemannian manifolds with global nonpositive curvature that offer convenient grounds for the generalisation of convexity notions. Specifically, we derive bounds on the expected instantaneous tracking error, and we provide algorithm parameter values that minimise the algorithm's performance. Our results illustrate how the manifold geometry in terms of the sectional curvature affects these bounds. Additionally, we provide dynamic regret bounds for this online optimisation setting. To the best of our knowledge, these are the first regret bounds even for the Euclidean version of the problem.
Lastly, via numerical simulations, we demonstrate the applicability of our algorithm on an online Karcher mean problem.
\end{abstract}

\section{Introduction}\label{sec:intro}
Time-varying optimisation problems are popular in the machine learning community under the framework known as \emph{online convex optimisation} (OCO) \cite{hazan2019introduction}. OCO is a promising methodology for modelling sequential tasks, and the main goal is developing algorithms that can track trajectories of the optimisers of the time-varying optimisation problem (up to asymptotic error bounds). OCO can be regarded as an iterative game between a player and an adversary. At each iteration $k\in\N_0$, the player (algorithm) selects a decision $x_k$ from a convex set $\mathcal{K}$, and the adversary reveals a convex function $f_k:\mathcal{K}\rightarrow\mathbb{R}$. The player subsequently suffers an instantaneous loss $f_k(x_k)$. Particularly, in this time-varying context, it is often of interest to solve the following sequence of optimisation problems,
\begin{align}\label{eq:OCO}
	\min_{x\in\mathcal{K}\subset\R^n} f_k(x).
\end{align}
We note that $k$ captures the time-varying nature of the underlying problem in the sense that between iterations, the revealed/sampled cost function varies, which is not the case in standard convex optimisation literature, where $f_k=f$ for all $k\in\N_0$. This feature makes OCO algorithms an appealing candidate to tackle dynamic optimisation tasks across many engineering and science domains such as power systems, robotics, and transportation systems, which inherit time variability in the optimisation problem at hand \cite{simonetto2020time}. OCO in the Euclidean setting \eqref{eq:OCO} has been vastly studied in the literature, and we refer the reader to \cite{simonetto2020time} for a review of available algorithms and applications that fit this problem structure.

Recently, optimisation over Riemannian manifolds has received a lot of attention given its applications to machine learning \cite{zhang2016riemannian}, signal processing \cite{manton2002optimization}, dictionary learning \cite{sun2016complete}, and low-rank matrix completion \cite{vandereycken2013low}, among others. However, time-varying optimisation problems in the Riemannian setting are much less developed. A relevant work in this context is \cite{feppon2019extrinsic}, where the authors provide a unified view for continuous matrix algorithms with the aim of tracking the value of some algebraic map. They show that each map can be geometrically interpreted as a projection from an ambient Euclidean space of matrices onto a matrix submanifold. These type of algorithms have applications in  data assimilation, data processing, machine learning, and matrix completion. On the other hand, numerical continuation methods solve questions that are related to the online optimisation setting. Recently, \cite{seguin2021continuation} studied numerical continuation methods in the context of Riemannian optimisation, where the authors develop and analyse a path-following numerical continuation algorithm on manifolds for solving the resulting parameter-dependent equation. The methods are then illustrated in two classical applications of Riemannian optimisation: the computation of the Karcher mean and low-matrix completion. 
Lastly, \cite{sato2019riemannian} studies online optimal identification of linear continuous systems. The identification problem is formulated as an optimisation problem on a Riemannian manifold, and the systems matrices to be identified vary over time.

In this paper, we aim to further expand on the above works and thus develop results in the area of time-varying Riemannian optimisation. Particularly, we focus on extending the OCO setting \eqref{eq:OCO} to the Riemannian case. A Riemannian manifold that provides proper grounds for the generalisation of convexity notions from the Euclidean setting to non-linear spaces is the Hadamard manifold \cite{bacak2014convex}. Technically speaking, Hadamard manifolds are Riemannian manifolds that are complete, simply connected, and with global nonpositive curvature \cite{bishop1969manifolds}---we define these concepts formally in Section \ref{sec:preliminaries}. Some classical examples include hyperbolic spaces, manifolds of positive definite matrices, and $\R^n$. Hadamard manifolds have applications in Brownian motion \cite{grigor2009volume}, Bayesian interference \cite{said2021bayesian}, online system identification \cite{sato2019riemannian}, diffusion tensor imaging \cite{arsigny2007geometric}, computational geometry \cite{onishi2002voronoi,nielsen2010hyperbolic}, medical imaging \cite{pennec2006riemannian}, and computer vision \cite{tuzel2008pedestrian,dong2014target}. In fact, many of these applications have inherent time-varying behaviour. For instance, in online system identification \cite{sato2019riemannian}, the cost functions are naturally time-varying since they use online data from a dynamical system. Moreover, in imaging for computer vision \cite{tuzel2008pedestrian}, robust principal component analysis (PCA)  is used to reduce outliers and noise \cite{bouwmans2018applications}. In this context, the measurement matrices are constructed from a window of video frames, and sliding the window over time yields to a time-varying optimisation problem. Lastly, as we illustrate in our example further below, in medical imaging applications \cite{fletcher2007riemannian} is often of interest to compute the Karcher mean or Riemannian centre of mass, which leads to a time-varying optimisation problem over Hadamard manifolds if the matrix measurements come from a moving object.

Consequently, studying time-varying optimisation problems over Hadamard manifolds is relevant and it is the main topic of interest in this paper. Particularly, the above applications lead to optimisation problems that can be generally formulated as follows,
\begin{align}\label{eq:problem}
	\min_{x\in\mathcal{X}\subset\M} f_k(x),
\end{align}
where $\M$ is a Hadamard submanifold embedded in $\R^n$, $\mathcal{X}$ is a closed subset of a geodesically convex set of $\mathcal{M}$, and each $f_k:\mathcal{M}\to\R$, $k\in\N_0$, is geodesically $L$-smooth and geodesically strongly convex (see Definitions \ref{def:smoothness} and \ref{def:strong-convexity} further below). Note that \eqref{eq:problem} is the extension of the Euclidean setting \eqref{eq:OCO} to the Riemannian case.

It is not uncommon that, in the aforementioned applications, explicit expressions for the cost functions $f_k$ may not be available, or their gradients are too costly to compute. For instance, in online system identification \cite{sato2019riemannian}, $f_k$ is a complicated function with respect to $x$ that depends on the output of a dynamical system, and it is evaluated through a real-world	experiment. Moreover, in  robust space tracking \cite{dixit2019online}, when the problem is large-scale, evaluating the full gradient at every iteration can be excessively costly since the cost function in such problems is often expressible as a sum of several component functions, each depending only on a subset of measurements. 
Therefore, we consider a gradient-free setting where only function evaluations can be obtained via an oracle. Each cost function is thus seen as a black-box with time-varying input-output map. 
The zeroth-order setting has been vastly studied in the Euclidean literature since it can be found in many applications where derivatives are either unavailable, or too expensive to compute, see \eg \cite{mania2018simple,malik2019derivative} for machine learning, \cite{ira2020tuning} for online controller tuning, \cite{chen2017zoo} for deep neural networks, and \cite{chen2018bandit} for mobile fog computing.

Formally, we aim to generate solutions to \eqref{eq:problem} using random gradient-free iterates of the form
\begin{align}\label{eq:iterate}
	x_{k+1} = \mathcal{P}_{\mathcal{X}}\left[\Exp_{x_k}(-\alpha_k g_{\eta,k^+}(x_k,u_k))\right],
\end{align}
where $\mathcal{P}_{\mathcal{X}}$ denotes the projection that maps a point $x\in\M$ to $\mathcal{P}_{\mathcal{X}}(x)\in\mathcal{X}\subset\M$ such that $\mathrm{dist}(x,\mathcal{P}_{\mathcal{X}}(x))<\mathrm{dist}(x,y)$, for all $y\in\mathcal{X} \backslash \{\mathcal{P}_{\mathcal{X}}(x)\}$, and $\mathcal{P}_{\X}(x)=x$ for $x\in\X$. The positive constant $\alpha_k$ denotes the step size, $g_{\eta,k^+}$ is the oracle or \emph{gradient estimator}, and $\Exp_{x_k}(\cdot)$ denotes the exponential mapping which we formally define in Section \ref{sec:preliminaries} below. We consider an extension of the recently proposed zeroth-order oracle for optimisation over Riemannian manifolds in \cite{li2020stochastic} to make it suitable for our time-varying setting \eqref{eq:problem}.  Particularly, we define the oracle as
\begin{align}\label{eq:oracle}
	g_{\eta,k^+}(x,u) \coloneqq \frac{f_{k^+}(\Exp_x(\eta u)) - f_k(x)}{\eta}u,
\end{align}
where $k^+\coloneqq k+1/2$, $k\in\N_0$ and $\eta>0$ corresponds to the oracle's precision. Some important discussions are in place behind the subscript $k^+$ and also the way $u$ is constructed. We note that the oracle is essentially a two-point estimate of the directional derivative, and the subscript $k^+$ captures the time-varying nature of the underlying process. That is, our framework permits the cost function to change between the two evaluations. For example, consider an estimation problem where there is a continuously varying dynamical system, so from the moment we sample the cost function at $x$ to the moment we evaluate it at $\Exp_x(\eta u)$, the cost function would have already changed by the underlying process, and this change is represented by $k^+$. For the forthcoming analysis, let us define $\mathcal{F}\coloneqq \{f_k:\mathcal{M}\to\mathbb{R}|k\in \N_0\cup \{j+1/2|j\in\N_0\}\}$.

Now, with respect to how $u$ is constructed, we follow the same approach as \cite{li2020stochastic}. That is, we let $u = Pu_0\in T_x\M$, where $u_0$ is normally distributed as per $u_0\sim\mathcal{N}(0,I_n)\in\R^n$, and $P\in\R^{n\times n}$ is the orthogonal projection matrix onto the tangent space $T_x\M$ of $\M$ at $x$.

\begin{remark}\label{rem:submanifold}
	Note that we assume that the cost functions $f_k$ are defined on submanifolds embedded in the Euclidean space, which makes it efficient to sample $u$ from the associated tangent space. As also discussed in \cite{li2020stochastic}, we emphasise that the above gradient estimation methodology could be generally applied to other manifolds; however, efficiently sample Gaussian random objects on the tangent space of general manifolds is not an easy task \cite{hsu2002stochastic}. Future work will focus on relaxing the way we sample the random vector $u$ by using results such as the ones in \cite{said2017riemannian}.
\end{remark}

\subsection{More related work}
In the general field of optimisation over Riemannian manifolds, and for the time-invariant---or \emph{offline}---setting, several results have been proposed in the literature, and we briefly review them  below. For instance, \cite{boumal2019global} provided convergence rates for deterministic Riemannian gradient descent and smooth cost functions. Stochastic algorithms were also considered for smooth Riemannian optimisation in \cite{bonnabel2013stochastic,zhang2016riemannian,kasai2018riemannian,zhou2019faster,weber2019projection}. Particularly, \cite{bonnabel2013stochastic} extended the classical stochastic gradient descent algorithms to the Riemannian case, and provided convergence results. The authors in \cite{zhang2016riemannian} introduced the variance--reduced \textsc{RSVRG} method and considered Riemannian optimisation of finite sums of geodesically smooth functions. The work \cite{kasai2018riemannian} proposed a Riemannian stochastic recursive gradient algorithm (\textsc{RSRG}) which provides notable computational advantages in comparison to \textsc{RSVRG}. The work \cite{zhou2019faster} introduced the Riemannian \textsc{SPIDER} method for non-convex Riemannian optimisation as a simple and efficient extension of the Euclidean \textsc{SPIDER} counterpart. Lastly, \cite{weber2019projection} studied stochastic projection-free methods for constrained optimisation of smooth functions on Riemannian manifolds. The stochastic Riemannian Frank-Wolfe methods for nonconvex and geodesically convex problems are introduced. For non-smooth cost functions, Riemannian subgradient methods have been proposed in \cite{li2019weakly}, manifold \textsc{ADMM} methods in \cite{kovnatsky2016madmm}, manifold proximal gradient (\textsc{ManPG}) methods in \cite{chen2020proximal}, manifold proximal point algorithms (\textsc{ManPPA}) in \cite{chen2019manifold}, and stochastic \textsc{ManPG} in \cite{wang2020riemannian}.

None of the aforementioned works have considered the zeroth-order setting, in which the oracle makes available only cost function values as opposed to first-order or second-order information. To the best of our knowledge, zeroth-order Riemannian optimisation for time-invariant cost functions have been considered in \cite{chattopadhyay2015derivative,dreisigmeyer2018direct,fong2019stochastic,li2020stochastic}. 
Particularly, \cite{fong2019stochastic} presented the extended Riemannian stochastic derivative-free optimisation (\textsc{RSDFO}) algorithm, and proved it converges in finitely many steps in compact connected Riemannian manifolds. The authors in \cite{chattopadhyay2015derivative} extended the derivative-free optimisation method by \cite{powell1964efficient} to Riemannian manifolds, but did not provide any complexity or convergence results. Just recently, \cite{li2020stochastic} provided the first complexity results for both deterministic and stochastic zeroth-order Riemannian optimisation. Their zeroth-order methods rely on an estimator of the Riemannian gradient based on a modification of the Gaussian smoothing technique from the seminal work by Nesterov \cite{nesspo17}. In \cite{li2020stochastic}, the authors illustrated that the proposed zeroth-order method has a comparable performance to its first-order counterparts in many applications of interest such as matrix approximation, k-PCA, sparse PCA, and the Karcher mean problem. 
Lastly, Dreisigmeyer in \cite{dreisigmeyer2018direct} studies direct search methods over general reductive homogenous spaces using maps from the tangent space to the manifold. We emphasise that these direct search methods use multiple mesh points per iteration, which is different to this paper, since we only evaluate the cost function twice at each step and take one action only.

We note that the existing works on zeroth-order Riemannian optimisation listed above do not consider the online setting \eqref{eq:problem} in which the cost-function is allowed to be time-varying. This problem, however, has been widely studied for $\M\equiv\R^n$, see e.g. \cite{bubeck2012regret,chiang2013beating,besbes2015non,yang2016tracking,dixit2019online}, in which it is assumed that cost function evaluations are carried out simultaneously when computing two-point estimates of the gradient. Later on, \cite{shames2019online} relaxed this assumption and allowed the cost function to change between function evaluations, which added an extra modelling layer that better respects the time-varying nature of the problem.

\subsection{Contributions}
Our contributions are threefold: 
\begin{itemize}
	\item We extend the OCO framework presented in \cite{shames2019online} from the Euclidean setting to the case where the cost function is defined on a Hadamard manifold. Our proposed algorithm uses an extension of the zeroth-order oracle recently presented in \cite{li2020stochastic} that allows for the function to be time-varying.	
	\item We provide asymptotic bounds on the expected instantaneous tracking error, which to the best of our knowledge, are the first error bounds for online zeroth-order optimisation on Riemannian manifolds available in the literature. Our results illustrate how the manifold geometry---in terms of the sectional curvature---influences the performance of the algorithm. In addition, we provide explicit choices for the algorithm parameters---step size and oracle's precision---that  minimise the performance of the algorithm. 
	\item Lastly, we provide dynamic regret bounds for our time-varying setting. These are the first available dynamic regret bounds for \eqref{eq:problem}, which were not previously available even for the Euclidean case.
\end{itemize}

\section{Preliminaries}\label{sec:preliminaries}
In this section, we present a brief introduction on the basics of manifold optimization. For a more in-depth revision we refer the reader to \cite{absil2009optimization}.
A smooth manifold is a topological manifold $\M$ with a globally defined differentiable structure. At any point $x$ on a smooth manifold, tangent vectors are defined as the tangents of parametrised curves passing through $x$. The tangent space $T_x\M$ of a manifold $\M$ at $x$ is defined as the set of all tangent vectors at the point $x$. Tangent vectors on manifolds generalise the notion of a directional derivative. Formally, we can define the tangent space $T_x\M$ as follows
\begin{align*}
	T_x\M \coloneqq \left\{\gamma'(0): \gamma(0)=x,\ \gamma([-\varepsilon,\varepsilon])\subset \M\ \text{for some}\ \varepsilon >0,\ \gamma\ \text{is differentiable} \right\}. 
\end{align*}
The tangent bundle of a differentiable manifold $\M$ is the manifold $T\M$ that assembles all the tangent vectors in $\M$. As a set, it is the disjoint union of all the tangent planes, \ie $T\M \coloneqq \sqcup_{x\in\M} T_x\M$.
The dimension of a manifold $\M$, denoted	as $d$, is the dimension of the Euclidean space that the manifold is locally homeomorphic to. In particular, the dimension of the tangent space is always equal to the dimension of the manifold.

A Riemannian manifold is a couple $(\M,g)$, where $\M$ is a smooth manifold equipped with a smoothly varying inner product (Riemannian metric) on the tangent space at every point, \ie $g(\cdot,\cdot)\coloneqq \inner{\cdot,\cdot}_x:T_x\M\times T_x\M\rightarrow \R$.	Without loss of generality, when the Riemannian
metric is clear from the context, we simply talk
about ``the Riemannian manifold $\M$''. Throughout, we assume the Levi-Civita connection is associated with $(\M,g)$.

\begin{definition}[Riemannian gradient]\label{def:gradient}
	Suppose $f$ is a smooth function on the Riemannian manifold $\M$. The Riemannian gradient $\grad f(x)$ is defined as the unique element of $T_x\M$ satisfying $\inner{\xi,\grad f(x)}_x = \left. \frac{\mathrm{d}}{\mathrm{d}t}f(\gamma(t))\right|_{t=0}$
	for any $\xi\in T_x\M$, where $\gamma(t)$ is a curve in $\M$ such that $\gamma(0)=x$ and $\gamma'(0)=\xi$.
\end{definition}

A Riemannian submanifold $\M$ of a Riemanninan manifold $\mathcal{N}$ is a submanifold of $\mathcal{N}$ equipped with the Riemannian metric inherited from $\mathcal{N}$. Since in this paper we assume that $\M$ is a Riemannian submanifold embedded in $\R^n$, $\M$ is equipped with the Riemannian metric inherited from $\R^n$. We thus write the inner product on the tangent space $\inner{\cdot,\cdot}_x$ at every point $x\in\M$ as $\inner{\cdot,\cdot}_x=\inner{\cdot,\cdot}$, with the right-hand side being the Euclidean inner-product. Consequently, the Riemannian gradient in Definition \ref{def:gradient} becomes  the projection of its Euclidean gradient onto the tangent space, that is $\grad f(x) = \mathcal{P}_{T_x\M}(\nabla f(x))$, where $\nabla f(x)$ denotes the Euclidean gradient of $f$ at $x$.

We now introduce a family of local parametrisations often called \emph{retractions}. In a nutshell, retractions allow us to move on a manifold (\ie move in the direction of a tangent vector) while staying on the manifold. 
Formally, we say that a retraction mapping $R_x$ is a smooth mapping from $T_x\M$ to $\M$ such that $R_x(0)=x$, and the differential at 0 is an identity mapping, \ie $\left.\frac{\mathrm{d}}{\mathrm{d}t}R_x(t\xi)\right|_{t=0}=\xi$, for all $\xi\in T_x\M$. We refer to the latter as the local rigidity condition. 
In other words, for every tangent vector $\xi\in T_x\M$, the curve $\gamma_\xi:t\mapsto R_x(t\xi)$ satisfies $\dot{\gamma}_\xi(0)=\xi$. Moving along the curve $\gamma_\xi$ is thought of as moving in the direction $\xi$ while constrained to the manifold $\M$.
In particular, the exponential mapping $\mathrm{Exp}_x$ is a retraction that generates geodesics. A geodesic is a curve representing in some sense the shortest path between two points in a Riemannian manifold. 

Throughout this paper, we assume that $\M$ is a Hadamard manifold, as introduced in Section \ref{sec:intro}, i.e.  complete, simply connected, and has nonpositive sectional curvature everywhere. Complete refers to the domain of the exponential mapping being the whole tangent bundle $T\M$, and simply connected means there are no circular paths that cannot be shrunk to a point. Hadamard manifolds have strong properties. For instance, there exists a unique geodesic between any two points on $\M$, and the exponential map is globally invertible at any point, $\Exp^{-1}_x:\M\rightarrow T_x\M$. The geodesic distance is thus given by $\mathrm{dist}(x,y)=\norm{\Exp^{-1}_x(y)}=\norm{\Exp^{-1}_y(x)}$, where $\norm{\cdot}$ is the norm associated with the Riemannian metric, which in our setting corresponds to the Euclidean norm as discussed above. On a Hadamard manifold, the notion of parallel transport provides a way to transport a vector along a geodesic. It is defined as the operator $\Gamma_x^y: T_x\M \rightarrow T_y\M$ which maps $v\in T_x\M$ to $\Gamma_x^y(v)\in T_y\M$ while preserving the inner product, \ie $\inner{u,v}_x = \inner{\Gamma_x^y (u), \Gamma_x^y (v)}_y$.

We introduce some important definitions.

\begin{definition}[Geodesically $L$-smoothness]\label{def:smoothness}
	A differentiable function $f\in\mathcal{F}$ is said to be geodesically $L$-smooth if there exists $L\geq 0$ such that the following inequality holds for all $x,y\in\M$,
	\begin{align}\label{eq:L-smoothness-gradient}
		\norm{\grad f(x) - \Gamma_y^x (\grad f(y))} \leq L \mathrm{dist}(x,y),
	\end{align}
	where we recall that $\mathrm{dist}(x,y)$ denotes the geodesic distance between $x$ and $y$, and $\Gamma_y^x$ is the parallel transport from $y$ to $x$.	
\end{definition}
It can be shown that if $f$ is geodesically $L$-smooth, then for any $x,y\in\M$ we have \cite{zhang2016first,zhou2019faster}
\begin{align}\label{eq:L-smoothness}
	f(y) \leq f(x) + \inner{\grad f(x),\Exp_x^{-1}(y)}_x + \frac{L}{2}\mathrm{dist}(x,y)^2.
\end{align}

We now provide convexity definitions on Hadamard manifolds.  We emphasize that these definitions would be more subtle for general Riemannian manifolds, which can have several geodesics between two points, see e.g. \cite{udriste2013convex}.
\begin{definition}[Geodesically convex set]
	A set $\mathcal{X}\subset \M$ is said to be geodesically convex if for any $x,y\in\X$, the unique shortest length geodesic connecting $x$ to $y$ lies entirely in $\X$.
\end{definition}

\begin{definition}
	A function $f:\X\to\R$ is said to be geodesically convex if for any $x,y\in\X$, 
	$f(\gamma(t))\leq (1-t)f(x) + tf(y)$, $\forall t\in[0,1]$, where $\gamma$ is the unique geodesic connecting $x$ to $y$.
\end{definition}

\begin{definition}[Geodesic strong convexity]\label{def:strong-convexity}
	A function $f:\X\to\R$ is said to be geodesically $\sigma$-strongly convex if there exists $\sigma\in(0,L]$ such that for any $x,y\in\X$,
	\begin{align*}
		f(y) \geq f(x) + \inner{\grad f(x),\mathrm{Exp}_x^{-1}(y)}_x + \frac{\sigma}{2} \mathrm{dist}(x,y)^2.
	\end{align*}
\end{definition}

The random object $u$ in \eqref{eq:oracle} is computed as in \cite{li2020stochastic}, i.e. $u=Pu_0$, where $u_0\sim\mathcal{N}(0,I_n)$ and $P$ is the orthogonal projection matrix onto the tangent space. Therefore, as in \cite{li2020stochastic}, we define
\begin{align*}
	\espu{g}\coloneqq \frac{1}{\nu}\int_{\R^n} g e^{-\frac{1}{2}\norm{u_0}^2du_0}	,
\end{align*}
where $\nu$ is the normalising constant, see Section 2.2 and Lemma 2.1 in \cite{li2020stochastic} for more details. Sometimes we will omit the subscript $u_0$ for simplicity.

\section{Online optimisation using zeroth-order Riemannian oracles}

Before presenting our results we state the underlying assumptions.

\begin{assumption}\label{assu:standing}
	\hspace{0cm}
	\begin{enumerate}
		\item[(a)] Every $f\in\mathcal{F}$ in \eqref{eq:problem} is geodesically $L$-smooth and geodesically strongly convex as per Definitions \ref{def:smoothness} and \ref{def:strong-convexity}, respectively.		
		\item[(b)] $\exists V\geq 0$ such that $\mathrm{dist}(x_{k+1}^\star,x_{k^+}^\star)\leq V$,		where $x_k^\star \coloneqq \arg\min_{x\in\M} f_k(x)$ for all $k\in\N_0$.
		\item[(c)] $\exists \delta>0$ such that $|f_k(x) - f_{k^+}(x)|\leq \delta$ for all $x\in\M$ and $k\in\N_0$.
		\item[(d)] The sectional curvature of $\M$ is lower-bounded by $\kappa\leq 0$.
		\item[(e)] $\exists R>0$ such that $\max_{y,z\in\mathcal{X}} \mathrm{dist}(y,z) \leq R$. 
	\end{enumerate}
\end{assumption}
The above assumptions are standard in the context of online convex optimisation, and have been vastly used in the literature when $\M\equiv \R^n$, see \eg \cite{shames2019online,dixit2019online,zhang2018dynamic,mokhtari2016online}.
Particularly, (a) generalises assumptions such as Lipschitz gradient smoothness and strong convexity (to ensure a unique minimiser) often adopted in convex optimisation over $\R^n$. Item (b) bounds the change in minimiser, also called \emph{path variation} in the literature \cite{yang2016tracking}, and (c) bounds the variation between consecutive cost functions. 
These assumptions on bounded variations essentially avoid two consecutive problems being arbitrarily different. For most tracking applications these assumptions are naturally satisfied, see e.g. \cite{dixit2019online,derenick2009optimal}. For instance, in the context of robust subspace tracking, they ensure that the underlying target subspace does not change abruptly \cite{vaswani2018robust}.
Note that  $\delta\to 0$ corresponds to the special case where the cost-function does not vary between evaluations and thus $g_{\eta,k^+}(x,u)=g_{\eta,k}(x,u)$ in \eqref{eq:oracle}. Lastly, (d) assumes a lower bound in the curvature of $\M$, and (e) an upper bound in the diameter of $\X$, which are common assumptions in Riemannian optimisation see \eg \cite{bonnabel2013stochastic,zhang2016first}, and \cite{li2020stochastic}.

In online convex optimisation, there are usually two measures of interest to assess performance of the algorithms: the \emph{tracking error} \cite{dixit2019online,shames2019online} and \emph{regret} \cite{hazan2019introduction}. Henceforth, we study both measures and thus provide performance guarantees for the proposed algorithm \eqref{eq:iterate} by means of upper bounds on both the tracking error and regret.
The tracking error essentially captures how well the algorithm \eqref{eq:iterate} can follow (or \emph{track}) the optimums of the time-varying optimisation problem \eqref{eq:problem} as $k$ grows. On the other hand, regret (or in this case \emph{dynamic regret} \cite{mokhtari2016online}), captures the accumulation of losses over the time horizon.

\subsection{Tracking error bounds}
The main objective of this section is to provide performance guarantees for algorithm \eqref{eq:iterate} by means of upper bounds on the tracking error. That is, characterise how well \eqref{eq:iterate} can track the optimisers of the time-varying problem \eqref{eq:problem}.
To that end, we define the \emph{tracking error} as
\begin{align*}
	e_k \coloneqq \mathrm{dist}(x_k,x_{k^+}^\star) = \norm{\Exp_{x_k}^{-1}(x_{k^+}^\star)} = \norm{\Exp_{x_{k^+}^\star}^{-1}(x_k)},
\end{align*}
and the \emph{estimation error} as
\begin{align*}
	\bar{e}_k \coloneqq \mathrm{dist}(x_{k+1},x_{k^+}^\star) = \norm{\Exp_{x_{k+1}}^{-1}(x_{k^+}^\star)} = \norm{\Exp_{x_{k^+}^\star}^{-1}(x_{k+1})}.
\end{align*}
The objects $e_k$ and $\bar{e}_k$ are also known as \emph{pre-update optimality gap} and \emph{post-update optimality gap}, respectively, see e.g. \cite{dixit2019online}.

To state our main tracking result, we require some intermediate steps which we present in the following.
First, a common technique for analysing optimisation algorithms is to write the estimation error in terms of the tracking error by using the law of cosines in the Euclidean space. Unfortunately, this equality does not exist for general non-linear spaces, and in fact, there are no corresponding analytical expressions. However, in \cite{zhang2016first}, a trigonometric distance bound for Alexandrov spaces with curvature bounded below was proposed. Alexandrov spaces are length spaces with curvature bound and form a generalisation of Riemannian manifolds with sectional curvature bounded below. The result uses the properties of \emph{geodesic triangles}, and it can be used as an analogue to the law of cosines given its fundamental nature. For our setting, this inequality is formalised in the lemma below.

\begin{lemma}\label{lem:triangle}
	For any Riemannian manifold $\M$ with a sectional curvature lower bounded by $\kappa$, and any points $x_{k^+}^\star,x_k\in\M$, the update \sloppy $x_{k+1} = \mathcal{P}_{\mathcal{X}}\left[\Exp_{x_k}(-\alpha_k g_{\eta,k^+}(x_k,u_k))\right]$ satisfies
	\begin{align}\label{eq:cosine}
		\bar{e}_k^2 \leq e_k^2 + 2\alpha_k \inner{g_{\eta,k^+}(x_k,u_k),\Exp^{-1}_{x_k}(x_{k^+}^\star)} + \zeta(\kappa,e_k) \alpha_k^2 \norm{g_{\eta,k^+}(x_k,u_k)}^2,
	\end{align} 
	where $\zeta(\kappa,e_k)\coloneqq e_k\sqrt{|\kappa|}/\tanh(e_k\sqrt{|\kappa|})$.
\end{lemma}	
\begin{proof}
	See Appendix \ref{sec:apendix-triangle}.
\end{proof} 

Note that for the Euclidean setting, we set the curvature $\kappa=0$ in Lemma \ref{lem:triangle}, which leads to $\zeta=1$.

The next intermediate step is to derive a relation between the conditional expectations of $e_{k+1}$ and $\bar{e}_k$. By Assumption \ref{assu:standing}(b) and the triangle inequality for the geodesic distance $\mathrm{dist}(\cdot,\cdot)$, we can write
\begin{align}
	e_{k+1}&= \mathrm{dist}(x_{k+1},x_{k^++1}^\star) \non  \\
	&\leq \mathrm{dist}(x_{k+1},x_{k^+}^\star) + \mathrm{dist}(x_{k^+}^\star,x_{k^++1}^\star) \non \\
	&\leq \bar{e}_k + 2V \non \\
	\espc{e_{k+1}}{x_k} &\leq \espc{\bar{e}_k}{x_k} + 2V. \label{eq:2V}
\end{align}

The last required intermediate step is the proposition below, where we present bounds related to the zeroth-order oracle that are essential to show our optimiser tracking result.
\begin{proposition}\label{propo:preliminaries}
	Under Assumption \ref{assu:standing}(a),(c), the following holds.
	\begin{enumerate}
		\item[(a)] $\displaystyle \norm{\espu{g_{\eta,k^+}(x,u)} - \grad f_{k^+}(x)} \leq \frac{L\eta}{2}(d+3)^{3/2} + \frac{\delta}{\eta} d^{1/2}$.
		\item[(b)] $\displaystyle \espu{\norm{g_{\eta,k^+}(x,u)}^2} \leq \frac{L^2\eta^2}{2} (d+6)^3 + 2L\delta (d+4)^2  + \frac{2\delta^2}{\eta^2}d +  2 (d+4)\norm{\grad f_{k^+}(x)}^2$.
	\end{enumerate}
\end{proposition}
\begin{proof} 
	See Appendix \ref{sec:apendix-B}.
\end{proof} 

Proposition \ref{propo:preliminaries} is the extension of Proposition 2.1 by \cite{li2020stochastic} to the time-varying case. As also noted in \cite{li2020stochastic} for the \emph{offline} case, we see that the oracle in \eqref{eq:oracle} is a biased estimator of the gradient in the online case, and the difference between them can be bounded as in Proposition \ref{propo:preliminaries}. Moreover, note that our bounds have extra terms (w.r.t. \cite{li2020stochastic}) that depend on $\delta$, which is the upper bound on the cost function variation by means of Assumption \ref{assu:standing}(c). We emphasise that Proposition \ref{propo:preliminaries} recovers the oracle bounds by \cite{li2020stochastic} for $\delta\to 0$ (time-invariant case).

We are now in a position to state a general result that illustrates how the conditional expectation of the tracking error evolves in time for any given step size $\alpha_k$. This is the main tool required to show that algorithm \eqref{eq:iterate} can track the optimisers of \eqref{eq:problem}.
\begin{theorem}\label{theo:main}
	Consider the iterates $x_{k+1} = \mathcal{P}_{\mathcal{X}}\left[\Exp_{x_k}(-\alpha_k g_{\eta,k^+}(x_k,u_k))\right]$ with $\alpha_k>0$ and $g_{\eta,k^+}$ as per \eqref{eq:oracle}. Then, under Assumption \ref{assu:standing}(a)--(d) we have that, for all $k\in\N_0$,
	\begin{align}\label{eq:theo1}
		\espp{e_{k+1} | x_k} \leq \sqrt{ \psi(e_k) } + 2V,
	\end{align}
	where 
	\begin{multline*}
		\psi(e_k) \coloneqq \left( 2(d+4)L^2\zeta(\kappa,e_k)\alpha_k^2 - \sigma \alpha_k + 1\right) e_k^2 \\
		+ \alpha_k \left(L\eta(d+3)^{3/2} + \frac{2\delta}{\eta} d^{1/2}\right) e_k \\
		+ \left( \frac{L^2\eta^2}{2} (d+6)^3 + 2L\delta (d+4)^2 + \frac{2\delta^2}{\eta^2} d \right) \zeta(\kappa,e_k)\alpha_k^2.
	\end{multline*}
\end{theorem}
\begin{proof}
	See Appendix \ref{sec:appendix-theo-main}.
\end{proof}

We can see that the conditional expectation of the tracking error depends on $\psi(e_k)$ which in turn depends on the parameters of the problem such as the Lipschitz constant $L$, manifold curvature $\kappa$, step size $\alpha_k$, oracle's precision $\eta$, and manifold dimension $d$. For instance, we can see that $\psi$ increases with $L$, however, its dependence on other parameters such as the step size and oracle's precision is not trivial. However, it turns out that if we pick a constant step size, we can obtain a simpler expression for the expected tracking error. Particularly, if we choose a constant step size in the interval $( 0 , \sigma/(2L^2(d+4)\zeta(\kappa,R))$, we can show that the expected tracking error remains bounded for $k\to\infty$, which is the main objective of this section. This is formalised in the below corollary.

\begin{corollary}\label{coro:delta}
	Under Assumption \ref{assu:standing}, if $\alpha_k = \alpha \in \left( 0 , \frac{\sigma}{2L^2(d+4)\zeta(\kappa,R)} \right)$ for all $k\in\N_0$, then
	\begin{align}\label{eq:limit}
		\underset{k\rightarrow\infty}{\mathrm{lim\, sup}}\ \espp{e_k} \leq \Delta \coloneqq \frac{D + 2V}{1 - \rho},
	\end{align}
	where $\rho \coloneqq \sqrt{2(d+4)L^2\zeta(\kappa,R)\alpha^2 - \sigma \alpha + 1}$, $D\coloneqq \alpha \max\{\theta_1,\theta_2\}$, and
	\begin{align*}
		\theta_1 &\coloneqq  \frac{L\eta(d+3)^{3/2}+(2/\eta)\delta d^{1/2}}{2\rho}, \\
		\theta_2 &\coloneqq \sqrt{\left( \frac{L^2\eta^2}{2} (d+6)^3 + 2L\delta (d+4)^2 + \frac{2\delta^2}{\eta^2} d \right) \zeta(\kappa,R)}\ .
	\end{align*}
\end{corollary}
\begin{proof} 
	See Appendinx \ref{sec:appendix-C}.
\end{proof}

Corollary \ref{coro:delta} shows that the expected value of the tracking error converges to a ball of radius $\Delta$, as long as we pick a constant step size $\alpha_k=\alpha$ in algorithm \eqref{eq:iterate} in the interval \sloppy $\alpha \in \left( 0 , \sigma/(2L^2(d+4)\zeta(\kappa,R))\right)$. That is, the algorithm will track the time-varying optimisers of \eqref{eq:problem} with an error of $\Delta$ as $k\to\infty$.

The upper bound $\Delta$ is used as a performance metric of the algorithm. Therefore, by minimising $\Delta$, we aim to improve the tracking performance of the algorithm.

\begin{remark}
	It is worth noticing that $\Delta$ depends on the manifold geometry through $\kappa$, as opposed to the Euclidean counterpart by \cite{shames2019online}. Moreover, $\Delta$ depends on the intrinsic dimension $d$ of the manifold, and not on the Euclidean ambient space dimension $n$ which could be considerably larger. This is due to the fact that we work directly on the manifold and perform appropriate extensions of notions such as geodesic convexity, exponential maps, etc. Therefore, it would be more costly to work in the larger ambient Euclidean space, and our Riemannian method should be used preferably, unless a specific structure in the larger space can be considerably exploited to simplify calculations. 
	The reported dependence of $\Delta$ on the manifold curvature is consistent with recent literature. For instance, in Riemannian \textsc{SVRG} algorithms, the convergence rate of the algorithm depends on the manifold curvature as observed by \cite{zhang2016riemannian}, see also \cite{zhang2016first} for similar conclusions on subgradient methods.
\end{remark}

\begin{remark}[Convergence of the algorithm]
	We can provide a complexity bound for the algorithm in terms of the number of iterations that takes the expected value of the tracking error to converge to a ball of radius $\Delta+\epsilon$, where $\epsilon$ is an arbitrary positive scalar. Particularly, from the proof of Corollary \ref{coro:delta}, we have that (see \eqref{eq:ek+1}),
	\begin{align*}
		\espp{e_{k+1}} \leq \rho\espp{e_k} + D + 2V,
	\end{align*}
	which leads to $\espp{e_k}\leq \rho^k\espp{e_0} + \frac{1-\rho^k}{1-\rho}(D+2V)$. With this inequality, we can show that there exists $K\in\N$ such that for all $k\geq K$, $\espp{e_k}\leq \Delta + \epsilon$. In fact, such $K$ satisfies
	\begin{align}\label{eq:K-complexity}
		K \leq \log\left[\frac{D+(1-\rho)\epsilon}{(1-\rho)\espp{e_0}-2V}\right](\log\rho)^{-1}.
	\end{align}
	It is assumed in \eqref{eq:K-complexity} that $(1-\rho)\espp{e_0}-2V>0$, otherwise $K=1$.
\end{remark}

In brief, we have shown that algorithm \eqref{eq:iterate} can track optimisers of the time-varying optimisation problem \eqref{eq:problem} up to an error of $\Delta$. We emphasise that this holds for any choice of constant step size in the interval $\alpha \in \left( 0 , \sigma/(2L^2(d+4)\zeta(\kappa,R))\right)$, and any choice of oracle's precision $\eta$. Indeed, 	$\Delta$ may be large depending on the underlying parameters. Therefore, it is crucial to find parameters that minimise $\Delta$, which is what we do in the following. Particularly, we find specific expressions for the step size $\alpha$ and oracle's precision $\eta$ such that the performance metric $\Delta$ is minimised.

\begin{theorem}\label{theo:optimal}
	Let $\bar{\eta}\coloneqq \left( 4\delta^2d/(L^2(d+6)^3) \right)^{1/4}$, and $\bar{\alpha}$ be the root\footnote{We note that the choice of step size $\bar{\alpha}$ in Theorem \ref{theo:optimal} always exists since $\Delta|_{\eta=\bar{\eta}}$ is convex in $\alpha$ over the interval $ \left(0,\tfrac{\sigma}{2L^2(d+4)\zeta(\kappa,R)}\right)$.} of $A\alpha^2+B\alpha+C=0$ in the interval $\Big( 0 , \frac{\sigma}{2L^2(d+4)\zeta(\kappa,R)} \Big)$ with
	\begin{align*}
		A &\coloneqq (8VL^2\zeta(\kappa,R)(d+4)+\sigma\bar{\theta})^2 - 8\bar{\theta}^2L^2\zeta(\kappa,R)(d+4) ,\\
		B &\coloneqq -4V\left( \bar{\theta}\sigma^2 + 8VL^2\zeta(\kappa,R)(d+4)\sigma + 8\bar{\theta}L^2\zeta(\kappa,R)(d+4) \right), \\
		C &\coloneqq (2\sigma V + 2\bar{\theta})^2 - 4\bar{\theta}^2 ,\\
		\bar{\theta} &\coloneqq \sqrt{\left( \frac{L^2\bar{\eta}^2}{2} (d+6)^3 + 2L\delta (d+4)^2 + \frac{2\delta^2}{\bar{\eta}^2} d \right) \zeta(\kappa,R)}.	
	\end{align*}
	Then, $\bar{\eta}$ and $\bar{\alpha}$ minimise $\Delta$ in \eqref{eq:limit}.
\end{theorem}
\begin{proof} 
	See Appendix \ref{sec:appendix-D}.
\end{proof} 

As a summary of the results of this section, we essentially stated that, if we use algorithm \eqref{eq:iterate} with choices of step size $\alpha$ and oracle's precision $\eta$ as per Theorem \ref{theo:optimal}, then the algorithm will track the optimisers of the time-varying optimisation problem \eqref{eq:problem} with performance $\Delta$, which is in fact optimal for this choice of algorithm parameters.

\subsection{Regret bounds}\label{sec:regret}
To cope with changing environments, another popular measure of performance for OCO algorithms is the so-called \emph{dynamic regret} \cite{hall2015online,yang2016tracking}, which compares the cumulative loss of the learner/player to a sequence of optimal solutions.
Consequently, we consider the following regret definitions, as counterparts to our tracking and estimation errors respectively,
\begin{subequations}\label{eq:regret-def}
	\begin{align}
		\reg_T^{\textrm{Track.}} &\coloneqq \sum_{k=0}^{T} \espp{f_{k^+}(x_k)}  - f_{k^+}(x_{k^+}^\star) , \\
		\reg_T^{\textrm{Est.}}   &\coloneqq \sum_{k=0}^{T} \espp{f_{k^+}(x_{k+1})}  - f_{k^+}(x_{k^+}^\star).
	\end{align}
\end{subequations}
The main goal is to achieve sublinear regret. We say that an algorithm performs well if its regret is sublinear as a function of $T$, since this implies that, on the average, the algorithm performs as well as a \emph{clairvoyant} who selects the minimiser at each step.

Before providing the dynamic regret bounds, we impose an extra assumption, which is standard in the literature for constrained OCO, see e.g. \cite{zhang2018dynamic,mokhtari2016online,gao2018online,hazan2020faster}.
\begin{assumption}\label{assu:gradient}
	For all $f\in\mathcal{F}$ and $x\in\X$, $\exists G > 0$ such that $\norm{\grad f(x)}\leq G$. 
\end{assumption}
Note that this assumption is widely used in convex optimisation, and it holds here since $\mathcal{X}$ is bounded and the cost functions are geodesically smooth.

For the above definitions of regret, we present the following bounds.
\begin{theorem}\label{theo:regret}
	If the step-size $\alpha_k$ and oracle's precision $\eta_k$ are chosen as
	\begin{subequations}\label{eq:alpha-eta}	
		\begin{align}
			0 < \alpha_k &< \min\left\{\textstyle \sqrt{ \frac{-\bar{\mathcal{B}}_k - (\bar{\mathcal{B}}_k^2 - 4 \bar{\mathcal{A}} \bar{\mathcal{C}}_k)^{\frac{1}{2}}}{2\bar{\mathcal{A}}} } , \frac{\sigma}{2L^2(d+4)\zeta(\kappa,R)} \right\}, \\
			\textstyle 0<\sqrt{ \frac{-\mathcal{B}_k - (\mathcal{B}_k^2 - 4 \mathcal{A} \mathcal{C})^{\frac{1}{2}}}{2\mathcal{A}} } \leq  \eta_k &\leq \textstyle  \sqrt{ \frac{-\mathcal{B}_k + (\mathcal{B}_k^2 - 4 \mathcal{A} \mathcal{C})^{\frac{1}{2}}}{2\mathcal{A}} },
		\end{align}
	\end{subequations}
	where
	\begin{align*}
		\bar{\mathcal{A}} &\coloneqq   4L^2\delta^2(d+4)^4\zeta^2 - 4L^2\delta^2d(d+6)^3 \zeta^2,\quad 
		\bar{\mathcal{B}}_k \coloneqq - \frac{4L\delta(d+4)^2\zeta \bar{c}^2}{T_k},\quad	\bar{\mathcal{C}}_k \coloneqq \frac{\bar{c}^4}{ T_k^2 } \\
		\mathcal{A} &\coloneqq \frac{L^2(d+6)^3\zeta}{2},\quad 
		\mathcal{B}_k \coloneqq 2L\delta(d+4)^2\zeta - \frac{\bar{c}^2}{\alpha_k^2 T_k},\quad
		\mathcal{C} \coloneqq 2\delta^2d\zeta,
	\end{align*}
	and $T_k = 2^m$ for $k\in[2^m-1,2^{m+1}-2]$, $m\in\N_0$. Then, the tracking and estimation regrets satisfy, for any $T\geq 1$,
	\begin{subequations}\label{eq:regret-bounds}
		\begin{align}
			\reg_T^{\textrm{Track.}} &\leq \tfrac{G}{1-\max\{\rho_0,\rho_T\}} \bigg( \espp{e_0} - \rho_T \espp{e_T} + \tfrac{\bar{c}\sqrt{2}}{1-\sqrt{2}}\sqrt{T} + V_T \bigg),\label{eq:regret-track}\\
			\reg_T^{\textrm{Est.}} &\leq \tfrac{G}{1-\max\{\rho_1,\rho_{T+1}\}}\bigg(  \espp{\bar{e}_0} - \rho_{T+1} \espp{\bar{e}_T} + \tfrac{\bar{c}\sqrt{2}}{1-\sqrt{2}}\sqrt{T} + \max\{\rho_1,\rho_{T+1}\} V_T\bigg), \label{eq:regret-est}
		\end{align}
	\end{subequations}
	where $\bar{c} >0$, $V_T\coloneqq \sum_{k=0}^{T-1} \mathrm{dist}(x_{k^+}^\star,x_{k^++1}^\star)$, and \sloppy $\rho_k\coloneqq \sqrt{2(d+4)L^2\zeta(\kappa,R)\alpha_k^2 - \sigma \alpha_k + 1}$.
\end{theorem}
\begin{proof} 
	See Appendix \ref{sec:appendix-regret-1}.
\end{proof}

Theorem \ref{theo:regret} provides a choice of step size $\alpha_k$ and oracle's precision $\eta_k$ such that the regret satisfies \eqref{eq:regret-bounds}. Note that $\alpha_k$ and $\eta_k$ are piecewise constant in periods of length $2^m$. For example, the periods $T_k$ have the form $T_0 = 1, T_1 = T_2 = 2, T_3=\cdots=T_6=4$, etc., and thus $(\alpha_1,\eta_1)=(\alpha_2,\eta_2)$, $(\alpha_3,\eta_3)=\cdots =(\alpha_6,\eta_6)$, and so on.

We have showed in \eqref{eq:regret-bounds} that, disregarding $V_T$, the proposed zeroth-order algorithm \eqref{eq:iterate} achieves sublinear regret, which highlights its performance. This error term $V_T$ relates to the change in minimisers (or path variation) given the time-varying nature of the problem at hand. Obviously, if $V_T$ is sub-linear in $T$, we can see that the presented regret bounds would be sub-linear in $T$ as well. 
It is therefore required that path variations diminish with $T$ to achieve sublinear regret. This holds if the target being tracked slows down over time or eventually stops \cite{hall2015online}.

\section{Numerical example}

To validate our results, we apply our zeroth-order algorithm to the problem of computing the Karcher mean of a collection of symmetric positive definite (SPD) matrices, also known as Riemannian centre of mass or Fr\'echet mean \cite{bini2013computing}. This problem appears in a number of applications such as medical imaging \cite{fletcher2007riemannian}, image segmentation \cite{rathi2007segmenting}, signal estimation \cite{kurtek2011signal}, and particle filtering \cite{bordin2018nonlinear}. Note that the Karcher mean is guaranteed to exist and be unique on a Hadamard manifold, see e.g. \cite{berger2012panoramic}. 

We consider a time-varying version of the Karcher mean problem which arises in online scenarios. For instance, we may want to find a central representative for a collection of online noisy measurements of a moving object. Formally, we consider that the measurements are $N$ SPD matrices of dimension $m\times m$ that become available at each time $k$, which we denote by $\{A_{k,1},\dots,A_{k,N}\}$.

The manifold of SPD matrices is defined as $\M\coloneqq \{ X\in\R^{m\times m}: X=X^\top \succ 0 \}$. If we equip $\M$ with the Riemannian metric
\begin{align*}
	\inner{M,N}_X \coloneqq \trace{X^{-1}M X^{-1}N},\quad M,N\in T_X\M,
\end{align*}
for every $X\in\M$, then the SPD manifold is a Hadamard manifold \cite{bacak2014convex}. The Riemannian distance is given by
\begin{align*}
	\mathrm{dist}(X,Y) \coloneqq \norm{\log\left( X^{-1/2}YX^{-1/2} \right)}_F,
\end{align*}
where $\norm{\cdot}_F$ corresponds to the Euclidean (or Frobenius) norm, and the exponential mapping is 
\begin{align*}
	\Exp_X(M) = X^{1/2}\mathrm{exp}\left( X^{-1/2}MX^{-1/2} \right)X^{1/2},\quad M\in T_X\M,
\end{align*}
for every $X\in\M$, where $\mathrm{exp}$ denotes the matrix exponential.
The time-varying cost function is defined as
\begin{align}\label{eq:cost}
	f_k(X) \coloneqq \frac{1}{2N} \sum_{i=1}^{N} \mathrm{dist}(X,A_{k,i})^2, \quad k\in\N_0.
\end{align}
The Karcher mean for each set of measurements $\{A_{k,i}\}_{i=1}^N$ received at time $k$ is the unique minimiser of $f_k(X)$, i.e. $x_k^\star \coloneqq \arg\min_{X\in\M} f_k(X)$, for all $k\in\N_0$.
The cost function \eqref{eq:cost} is known to be geodesically strongly convex with $\sigma=1$ and geodesically $\zeta$--smooth (i.e. \eqref{eq:L-smoothness-gradient} holds with $L=\zeta$), see e.g. \cite{zhang2016first} and \cite{zhang2016riemannian}. We consider $\delta=0.001$, $N=10$, two problem sizes $m\in\{3,9\}$, and the manifold dimension is $d=m(m+1)/2$. For this example, we estimated\footnote{
	We emphasise that the parameters $V$ and $\zeta$ were numerically bounded in this simulation. This is a common approach used in experimental optimisation methods, see e.g. \cite{bunfra16,ahmed2020combining}, where the parameters are found so that the underlying assumptions are at least satisfied in the simulation/experimental data.
} $V=0.5$ and $\zeta=1.5$.
The step size and oracle's precision for each $m\in\{3,9\}$ are chosen as per Theorem \ref{theo:optimal}, which gives $\bar{\alpha}_3 = 0.0074$, $\bar{\alpha}_9=0.0015$, $\bar{\eta}_3=0.0089$, and $\bar{\eta}_9=0.005$.
The matrices $\{A_{k,i}\}_{i=1}^N$ were randomly generated using the \textsc{Manopt} toolbox in \textsc{Matlab}, see \cite{boumal2014manopt}.

Figure \ref{fig:experiment} (left) depicts the average tracking error $\espp{e_k}$  after implementing the zeroth-order algorithm \eqref{eq:iterate} for 100 random runs and different values of problem size. 
We can see the error converges, and its asymptotic value is upper bounded by some $\Delta$. That is, in expectation, the algorithm can track the optimisers up to an asymptotic error. The theoretical asymptotic bounds from Corollary \ref{coro:delta} are $\Delta_3 = 543.73$ and $\Delta_9 = 2666$, for $m\in\{3,9\}$. Comparing with the asymptotic values in Figure \ref{fig:experiment}, we can conclude that these bounds can become conservative depending on the application. Obtaining less conservative bounds is an open question for future research.
It can also be seen that the larger the problem size, the longer $\espp{e_k}$ takes to converge.

Lastly, we have also included a comparison with the first-order version of \eqref{eq:iterate}, in which the gradient of the function is fully available at every iteration $k$. We note that this is for illustration purposes only, since the results of this paper assume only cost function evaluations are available. We note a difference in rate of convergence, as expected, but also in terms of the asymptotic $\Delta$ they achieve. We can see that if we have access to the gradient, then $\Delta$ is smaller. This can actually be interpreted from the theory developed above. Particularly,
from Proposition \ref{propo:preliminaries}, we know the oracle is a biased estimator of the gradient of the cost function, and this difference can be upper bounded. These error terms wound not appear if the gradient is available at every $k$, and thus $\Delta$ in Corollary \ref{coro:delta} would be larger for the zeroth-order case. This, in link with Remark \ref{rem:submanifold}, can lead to interesting future work related to the choice of oracles and the way $u$ is sampled when computing them.

\begin{remark} 
	In the context of our example, it is worth mentioning that for applications such as diffusion tensor imaging, the Euclidean averaging of SPD matrices often leads to a ``swelling effect'', artificial extra diffusion introduced in computation, see e.g. \cite{arsigny2007geometric}. Particularly, it means the determinant of the Euclidean mean can be strictly larger than the original determinants. In diffusion tensor imaging, diffusion tensors correspond to covariance matrices of the local Brownian motion of water molecules. Introducing more diffusion is physically unacceptable in this context. Therefore, Riemannian approaches such as the one presented in this paper would be preferred. 
\end{remark}

\begin{figure}
	\centering 
	\begin{tabular}{cc}
		\includegraphics[scale=0.55]{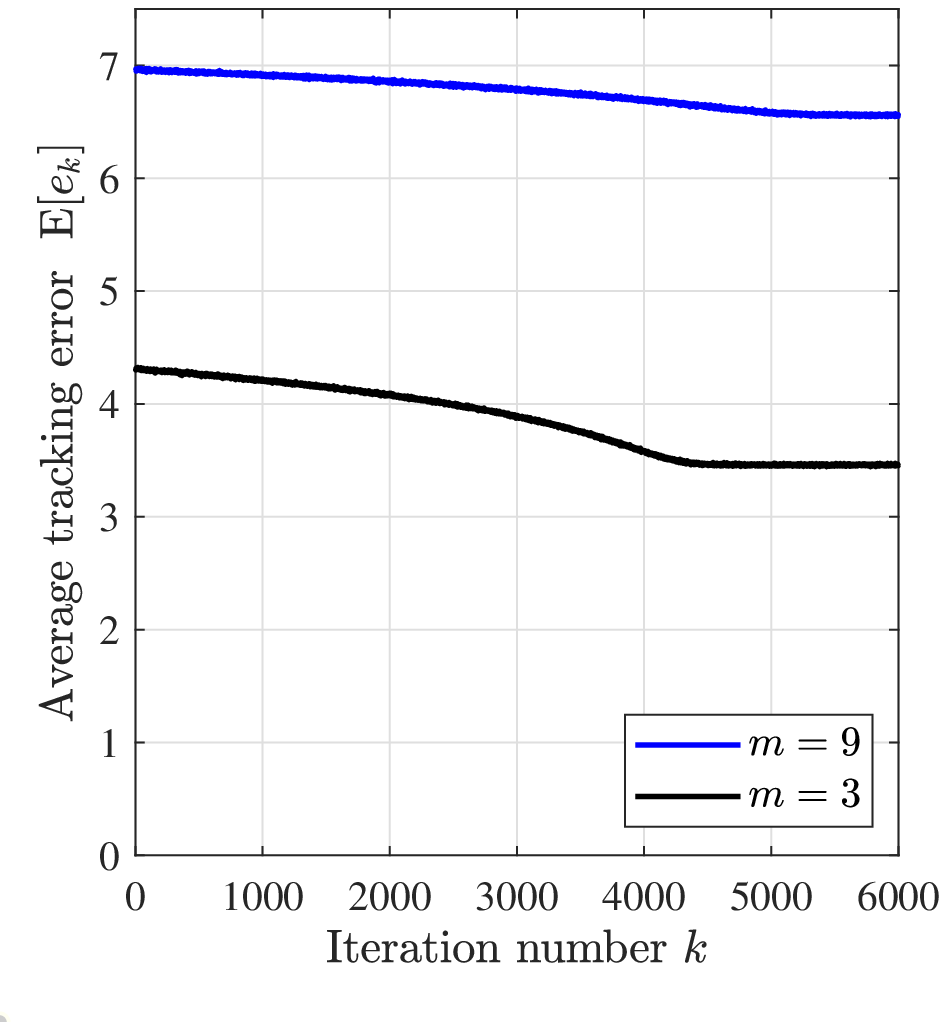} &
		\includegraphics[scale=0.55]{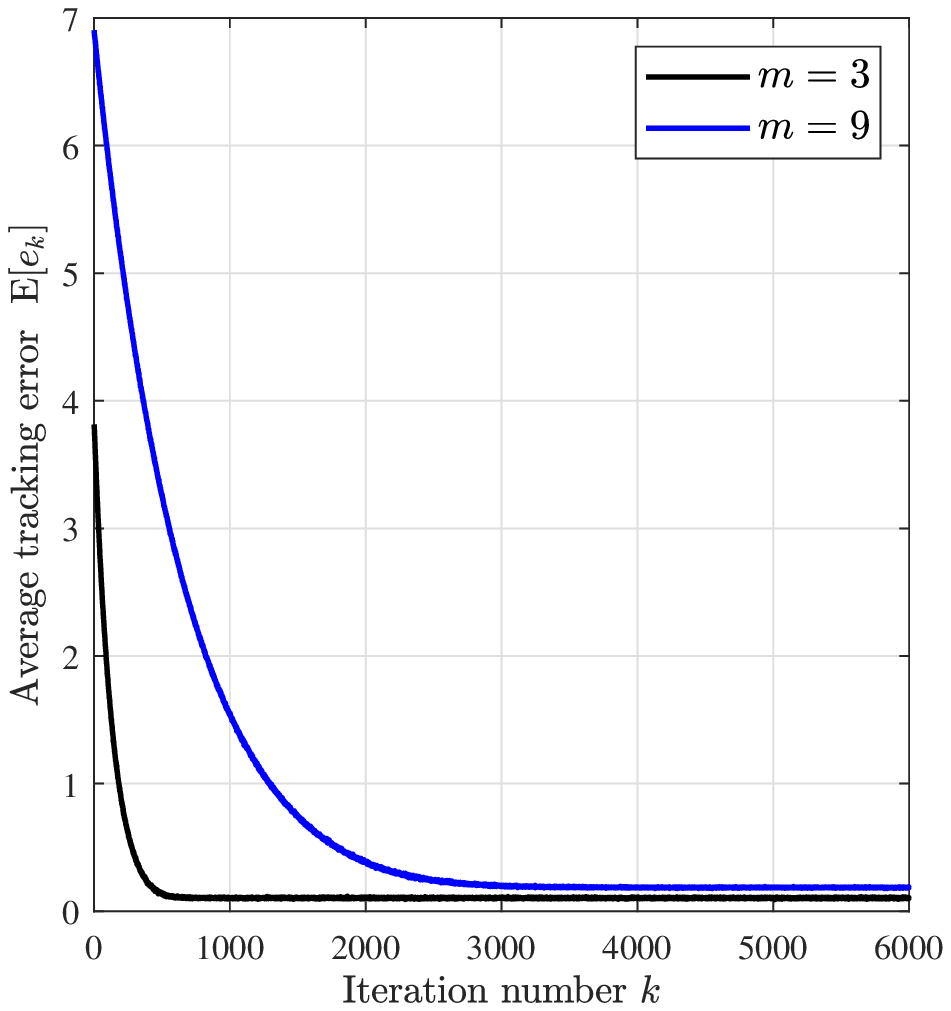}
	\end{tabular}
	\caption{Algorithm's average tracking error---over 100 random scenarios---for different values of problem size $m\in\{3,9\}$ and for: (left) zeroth-order iterates \eqref{eq:iterate}, and (right) first-order iterates that use $\grad f_k$ instead of $g_{\eta,k}$ in \eqref{eq:iterate}.}
	\label{fig:experiment}
\end{figure}



\section{Conclusions}
A gradient-free algorithm for the minimisation of time-varying cost functions on Hadamard manifolds was proposed. Bounds on the expectation of the tracking error and on dynamic regret were derived, and choices for algorithm parameters  such that the asymptotic tracking error bound is minimised were provided. Finally, the theoretical results were validated via numerical experiments.

Future work includes the extension to a more general class of Riemannian manifolds by trying to relax the way we sample the random vector $u$ in \eqref{eq:oracle}, and also the extension to the non-convex time-varying case. In addition, looking at the problem in a different set of coordinates at each step is also an interesting future direction. Lastly, exploring different gradient approximations in the oracle for this time-varying context is also of interest.


\appendix

\section{Auxiliary lemmas}

\begin{lemma}[\cite{bacak2014convex}]\label{lem:nonexpansive}
	Let $\M$ be a Hadamard manifold and $\mathcal{X}\subset \M$ a closed convex set. Then, the mapping $\mathcal{P}_{\mathcal{X}}(x)\coloneqq \{y\in\mathcal{X}:\mathrm{dist}(x,z)=\inf_{z\in\mathcal{X}}\mathrm{dist}(x,z)\}$ is single-valued and nonexpansive, that is, we have $\mathrm{dist}(\mathcal{P}_{\mathcal{X}}(x),\mathcal{P}_{\mathcal{X}}(y)) \leq \mathrm{dist}(x,y)$ for every $x,y\in\M$.
\end{lemma}

\begin{lemma}[\cite{shames2019online}]\label{lem:quadratic}
	Let $x,y,a,c\geq 0$ and $b\in\R$. Then $x^2 \leq a y^2 + b y + c$ implies $x \leq y \sqrt{a} + D$,
	where $D \coloneqq \max\left\{ \frac{b}{2\sqrt{a}}, \sqrt{c}  \right\}$.
\end{lemma}

\begin{lemma}[\cite{li2020stochastic}]\label{lem:identities}
	Suppose $\mathcal{X}$ is a $d$--dimensional subspace of $\R^n$, with orthogonal projection matrix $P\in\R^{n\times n}$, $u_0\sim\mathcal{N}(0,I_n)$, and $u=P u_0$ is the orthogonal projection of $u_0$ onto $\mathcal{X}$. Then,
	\begin{enumerate}
		\item[(a)] $x = \frac{1}{\nu} \int_{\R^n} \inner{x,u} u e^{-\frac{1}{2}\norm{u_0}^2}du_0$, $\forall x\in\mathcal{X}$.
		\item[(b)] For $p\in[0,2]$, $\espu{\norm{u}^p}\leq d^{p/2}$, and if $p\geq 2$, then $\espu{\norm{u}^p}\leq (d+p)^{p/2}$.
		\item[(c)] $\espu{\norm{\inner{\grad f_{k^+}(x),u}u}^2} \leq (d+4)\norm{\grad f_{k^+}(x)}^2$.
	\end{enumerate}
\end{lemma}

\begin{lemma}[\cite{zhang2016first}]\label{lem:zhang}
	If $a,b,c$ are the sides (\ie lengths) of a geodesic triangle in an Alexandrov space with curvature lower bounded by $\kappa$, and $A$ is the angle between sides $b$ and $c$, then
	\begin{align*}
		a^2 \leq \frac{c\sqrt{|\kappa|}}{\tanh(c\sqrt{|\kappa|})} b^2 + c^2 - 2bc\cos(A).
	\end{align*}
\end{lemma} 

\section{Proof of Lemma \ref{lem:triangle}.}\label{sec:apendix-triangle}
Let $\tilde{x}_{k+1}\coloneqq \Exp_{x_k}(-\alpha_k g_{\eta,k^+}(x_k,u_k))$, and consider the geodesic triangle depicted in Figure \ref{fig:triangle} with vertices $x_{k^+}^\star, x_k,$ and $\tilde{x}_{k+1}$, and sides $a\coloneqq\mathrm{dist}(\tilde{x}_{k+1},x_{k^+}^\star)$, $b\coloneqq \mathrm{dist}(x_k,\tilde{x}_{k+1})$, and $c\coloneqq e_k = \mathrm{dist}(x_k,x_{k^+}^\star)=\norm{\Exp^{-1}_{x_k}(x_{k^+}^\star)}$.
For this triangle, we have that $\mathrm{dist}(x_k,\tilde{x}_{k+1})= \norm{\Exp^{-1}_{x_k}(\tilde{x}_{k+1})}=\alpha_k\norm{g_{\eta,k^+}(x_k,u_k)}$. In addition, we have that $$bc\cos(A)=\inner{-\alpha_kg_{\eta,k^+}(x_k,u_k),\Exp_{x_k}^{-1}(x_{k^+}^\star)}.$$	
Then, by Lemma \ref{lem:zhang},
\begin{align}\label{eq:zhang-proof}
	\mathrm{dist}(\tilde{x}_{k+1},x_{k^+}^\star)^2 \leq e_k^2 + 2\alpha_k \inner{g_{\eta,k^+}(x_k,u_k),\Exp^{-1}_{x_k}(x_{k^+}^\star)}  + \zeta(\kappa,e_k) \alpha_k^2 \norm{g_{\eta,k^+}(x_k,u_k)}^2 .
\end{align}
Lastly, note that by Lemma \ref{lem:nonexpansive}, $\mathrm{dist}(\tilde{x}_{k+1},x_{k^+}^\star)^2\geq \mathrm{dist}(x_{k+1},x_{k^+}^\star)^2=\bar{e}_k^2$, and thus the result follows immediately from \eqref{eq:zhang-proof}.

\begin{figure}
	\centering 
	\includegraphics[scale=0.5]{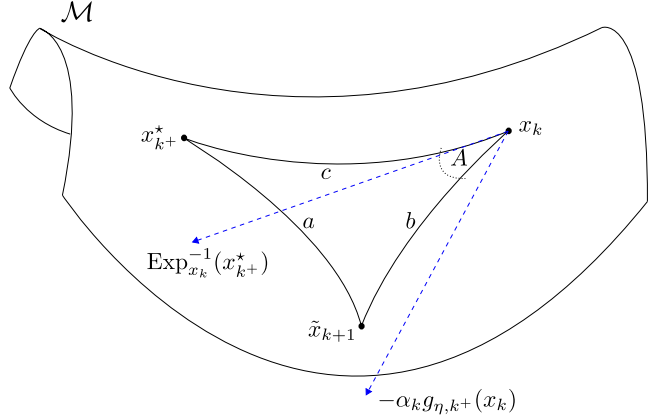}
	\caption{Illustration of the geodesic triangle used in Lemma \ref{lem:zhang}.}
	\label{fig:triangle}
\end{figure}

\section{Proof of Proposition \ref{propo:preliminaries}.}\label{sec:apendix-B}
\begin{enumerate}
	\item[(a)]  We complete the proof with the following steps. Essentially we want to quantify how well the expectation of the oracle approximates the real gradient. 
	\begin{align*}
		&\norm{ \espp{g_{\eta,k^+}(x,u)} - \grad f_{k^+}(x) } \\
		&\hspace{0.5cm}\stackrel{\text{Lemma \ref{lem:identities}(a)}}{=}\hspace{2mm} 
		\bigg\|  \frac{1}{\nu} \int_{\R^n} \bigg( 
		\frac{f_{k^+}(\Exp_x(\eta u)) - f_k(x)}{\eta}   - \inner{\grad  f_{k^+}(x),u } \bigg) u e^{-\frac{1}{2} \norm{u_0}^2}du_0 \bigg\| \\
		&\hspace{1.25cm}\leq \hspace{9mm} \frac{1}{\eta \nu} \int_{\R^n} \big|f_{k^+}(\Exp_x(\eta u)) - f_{k^+}(x) - \inner{\grad  f_{k^+}(x), \eta u } \\
		&\hspace{4cm} + f_{k^+}(x) - f_k(x)  \big| \norm{u} e^{-\frac{1}{2} \norm{u_0}^2}du_0 \\
		&\hspace{0.3cm}\stackrel{\text{Assum. \ref{assu:standing}(a),(c)}}{\leq} \frac{1}{\eta \nu} \int_{\R^n} \left(\frac{L\eta^2}{2}\norm{u}^2 + \delta \right)\norm{u} e^{-\frac{1}{2} \norm{u_0}^2}du_0 \\
		&\hspace{1.25cm}=\hspace{9mm} \frac{L\eta}{2\nu} \int_{\R^n} \norm{u}^3 e^{-\frac{1}{2} \norm{u_0}^2}du_0 + \frac{\delta}{\eta \nu }\int_{\R^n} \norm{u} e^{-\frac{1}{2}\norm{u_0}^2}du_0 \\
		&\hspace{0.5cm}\stackrel{\text{Lemma \ref{lem:identities}(b)}}{\leq}\hspace{2mm} \frac{L\eta}{2}(d+3)^{3/2} + \frac{\delta}{\eta} d^{1/2}.	
	\end{align*}
	\item[(b)]  To compute a bound on $\espp{\norm{g_{\eta,k^+}(x,u)}^2}$, we proceed by definition of $g_{\eta,k^+}(x,u)$, and thus first we bound $f_{k^+}(\Exp_x(\eta u)) - f_k(x)$.	
	We start by adding a convenient zero, that is, we add and subtract both $f_{k^+}(x)$ and $\inner{\grad f_{k^+}(x),\eta u}$, and then we use Assumptions \ref{assu:standing}(a) and \ref{assu:standing}\text{(c)}. That is,
	\begin{align*}
		( f_{k^+}(\Exp_x(\eta u)) - f_k(x) )^2 &= \big( f_{k^+}(\Exp_x(\eta u)) - f_{k^+}(x) - \inner{\grad f_{k^+}(x),\eta u} \\
		&\qquad + f_{k^+}(x) - f_k(x)
		+  \inner{\grad f_{k^+}(x),\eta u} \big)^2 \\
		&\leq \left( \frac{L\eta^2}{2}\norm{u}^2 + \delta + \inner{\grad f_{k^+}(x),\eta u} \right)^2 \\
		&\leq 2\left( \frac{L\eta^2}{2}\norm{u}^2 + \delta \right)^2 + 2\eta\inner{\grad f_{k^+}(x),u}^2.
	\end{align*}
\end{enumerate}
Consequently,
\begin{align*}
	\espp{\norm{g_{\eta,k^+}(x,u)}^2} &= \frac{(f_{k^+}(\Exp_x(\eta u)) - f_k(x))^2}{\eta^2}\espp{\norm{u}^2} \\ 
	&\leq \frac{L^2\eta^2}{2} \espp{\norm{u}^6} + 2L\delta \espp{\norm{u}^4} + \frac{2\delta^2}{\eta^2}\espp{\norm{u}^2}  + 2\espp{ \norm{\inner{\grad f_{k^+}(x),u} u}^2 }. 
\end{align*}
The proof is thus complete by means of Lemma \ref{lem:identities}(b),(c).

\section{Proof of Theorem \ref{theo:main}.}\label{sec:appendix-theo-main}
In this proof our main goal is to obtain \eqref{eq:theo1}. To achieve this, we note that \eqref{eq:2V} gives us a relation between $\espc{e_{k+1}}{x_k}$ and $\espc{\bar{e}_k}{x_k}$. Therefore, we first compute the following.
We take conditional expectations in \eqref{eq:cosine} and obtain
\begin{align}
	\espc{\bar{e}_k^2}{x_k} &\leq e_k^2 + 2\alpha_k \inner{\espc{g_{\eta,k^+}(x_k,u_k)}{x_k},\Exp^{-1}_{x_k}(x_{k^+}^\star)} \non \\
	&\qquad + \zeta(\kappa,e_k) \alpha_k^2 \espc{ \norm{g_{\eta,k^+}(x_k,u_k)}^2}{x_k} \non  \\
	&\leq e_k^2 + 2\alpha_k \inner{\espp{g_{\eta,k^+}(x_k,u_k)},\Exp^{-1}_{x_k}(x_{k^+}^\star)} \non \\
	&\qquad + \zeta(\kappa,e_k) \alpha_k^2 \bigg( \frac{L^2\eta^2}{2} (d+6)^3 + 2L\delta (d+4)^2 \non \\
	&\qquad + \frac{2\delta^2}{\eta^2} d + 2 (d+4)\norm{\grad f_{k^+}(x_k)}^2 \bigg), \label{eq:ebar}
\end{align}
where the last inequality follows from Proposition \ref{propo:preliminaries}(b).
Below, we focus on computing the term $\inner{\espp{g_{\eta,k^+}(x_k,u_k)},\Exp^{-1}_{x_k}(x_{k^+}^\star)}$ in \eqref{eq:ebar}. We add a convenient zero, in this case we add and subtract $\grad f_{k^+}(x_k)$
\begin{multline*}
	\inner{\espp{g_{\eta,k^+}(x_k,u_k)},\Exp^{-1}_{x_k}(x_{k^+}^\star)} = \inner{\espp{g_{\eta,k^+}(x_k,u_k)} - \grad f_{k^+}(x_k),\Exp^{-1}_{x_k}(x_{k^+}^\star)}\\
	+ \inner{ \grad f_{k^+}(x_k), \Exp^{-1}_{x_k}(x_{k^+}^\star)}.
\end{multline*}
Now, note that since $f$ is geodesically strongly convex as per Assumption \ref{assu:standing}(a), it is immediate to show that each $f$ satisfies
\begin{align}\label{eq:secant-inequality}
	-\inner{\grad f(x),\Exp_x^{-1}(x^\star)} \geq \frac{\sigma}{2} \mathrm{dist}(x,x^\star)^2,
\end{align}
for all $x\in\M$, where $x^\star \coloneqq \arg\min_{x\in\M} f(x)$. Inequality \eqref{eq:secant-inequality} is the Riemannian counterpart of the restricted secant inequality in $\R^n$.
Then, by the Cauchy-Scharwz inequality and \eqref{eq:secant-inequality},
\begin{align} 
	&\inner{\espp{g_{\eta,k^+}(x_k,u_k)},\Exp^{-1}_{x_k}(x_{k^+}^\star)} \non \\
	&\hspace{1cm}\leq   \norm{\espp{g_{\eta,k^+}(x_k,u_k)} - \grad f_{k^+}(x_k)}\norm{\Exp^{-1}_{x_k}(x_{k^+}^\star)}      - \frac{\sigma}{2} \mathrm{dist}(x_k,x_{k^+}^\star)^2 \non \\ 
	&\hspace{1cm}= \norm{\espp{g_{\eta,k^+}(x_k,u_k)} - \grad f_{k^+}(x_k)}e_k     - \frac{\sigma}{2} e_k^2 \non \\
	&\hspace{1cm}\leq \left(\frac{L\eta}{2}(d+3)^{3/2} + \frac{\delta}{\eta} d^{1/2}\right) e_k - \frac{\sigma}{2} e_k^2, \label{eq:inner}
\end{align}
where the last inequality follows from Proposition \ref{propo:preliminaries}(a).

There is now only one term in \eqref{eq:ebar} that remains to be bounded, which is $\norm{\grad f_{k^+}(x_k)}^2$. To that end, note that from \eqref{eq:L-smoothness-gradient} and the reverse triangle inequality, $\left| \norm{\grad f(x)} - \norm{\Gamma_y^x\grad f(y)} \right| \leq L \mathrm{dist}(x,y)$,
which, in turn, implies that $\norm{\grad f_{k^+}(x_k)} \leq L e_k$. Therefore, by using the latter together with \eqref{eq:inner} into \eqref{eq:ebar}, we obtain
\begin{align*}
	\espc{\bar{e}_k^2}{x_k} &\leq e_k^2 + 2\alpha_k \left(  \left(\frac{L\eta}{2}(d+3)^{3/2} + \frac{\delta}{\eta} d^{1/2}\right) e_k - \frac{\sigma}{2} e_k^2  \right)  \\
	&\qquad + \zeta(\kappa,e_k)\alpha_k^2 \bigg( \frac{L^2\eta^2}{2} (d+6)^3 + 2L\delta (d+4)^2 + \frac{2\delta^2}{\eta^2} d + 2 (d+4)L^2e_k^2 \bigg) \\
	&= \left( 2(d+4)L^2\zeta(\kappa,e_k)\alpha_k^2 - \sigma \alpha_k + 1\right) e_k^2 \\
	&\qquad + \alpha_k \left(L\eta(d+3)^{3/2} + \frac{2\delta}{\eta} d^{1/2}\right) e_k \\
	&\qquad + \left( \frac{L^2\eta^2}{2} (d+6)^3 + 2L\delta (d+4)^2 + \frac{2\delta^2}{\eta^2} d \right) \zeta(\kappa,e_k)\alpha_k^2 .
\end{align*}

By Jensen's inequality we get $\espc{\bar{e}_k}{x_k}^2 \leq \espc{\bar{e}_k^2}{x_k}$, and the proof is thus complete from applying \eqref{eq:2V}.

\section{Proof of Corollary \ref{coro:delta}.}\label{sec:appendix-C}
The first part of the proof consists in simplifying the expression for $\psi(e_k)$ that comes from Theorem \ref{theo:main} given the choice of constant step-size and also the bound on $\zeta$.
Note that Assumption \ref{assu:standing}(e) implies that $\zeta(\kappa,e_k)\leq \zeta(\kappa,R)$. Then, for $\alpha_k=\alpha$,
\begin{multline*}
	\psi(e_k) = \underbrace{ \left( 2(d+4)L^2\zeta(\kappa,R)\alpha^2 - \sigma \alpha + 1\right) }_{\coloneqq a}e_k^2 \\
	+ \underbrace{ \alpha \left(L\eta(d+3)^{3/2} + \frac{2\delta}{\eta} d^{1/2}\right)}_{\coloneqq b} e_k \\
	+ \underbrace{ \left( \frac{L^2\eta^2}{2} (d+6)^3 + 2L\delta (d+4)^2 + \frac{2\delta^2}{\eta^2} d \right) \zeta(\kappa,R)\alpha^2 }_{\coloneqq c}.
\end{multline*}
The second part of the proof boils down to using the above simplified expression for $\psi(e_k)$ to obtain a recursive expression for $\espp{e_k}$ so that we can iterate it and compute \eqref{eq:limit}. Recall from the proof of Theorem \ref{theo:main} that $\espc{\bar{e}_k}{x_k}^2 \leq \psi(e_k) = a e_k^2 + b e_k + c$. Therefore, Lemma \ref{lem:quadratic} implies $\espc{\bar{e}_k}{x_k} \leq \rho e_k + D$. We then use the latter inequality in \eqref{eq:2V} to get $\espc{e_{k+1}}{x_k}\leq \rho e_k + D + 2V$. Applying expectation then leads to the following recursive equation for $\espp{e_k}$,
\begin{align}\label{eq:ek+1}
	\espp{e_{k+1}}  \leq \rho \espp{e_k}  + D + 2V.
\end{align}
Since $\rho<1$ for $0 < \alpha < \frac{\sigma}{2L^2(d+4)\zeta(\kappa,R)}$, we can iterate \eqref{eq:ek+1} and obtain \eqref{eq:limit} in the limit $k\to\infty$, concluding the proof.

\section{Proof of Theorem \ref{theo:optimal}.}\label{sec:appendix-D}
The first part of the proof consists in showing that $\theta_2>\theta_1$, which will allow us to write $D=\alpha\max\{\theta_1,\theta_2\}=\alpha\theta_2$.
Note that $\rho^2 = 2(d+4)L^2\zeta(\kappa,R)\alpha^2 - \sigma\alpha + 1 > \frac{L^2}{2}\alpha^2 - \sigma\alpha + 1 \geq \frac{\sigma^2}{2}\alpha^2 - \sigma \alpha + 1$,
where the last inequality follows from Assumption \ref{assu:standing}(a). Then, $2\rho^2 > (\sigma\alpha)^2 - 2\sigma\alpha + 2 = (\sigma\alpha - 1)^2 + 1 \geq 1$,
which implies that $2\rho >\sqrt{2}$. Therefore,
\begin{align*}
	\theta_1 = \frac{L\eta(d+3)^{3/2}+(2/\eta)\delta d^{1/2}}{2\rho} <  \frac{L\eta(d+3)^{3/2}+(2/\eta)\delta d^{1/2}}{\sqrt{2}},
\end{align*}
and thus
\begin{align}\label{eq:theta_1}
	\theta_1^2 < \frac{L^2\eta^2}{2}(d+3)^3 + 2L\delta d^{1/2}(d+3)^{3/2} + \frac{2\delta^2}{\eta^2}d .
\end{align}
On the other hand,
\begin{align*}
	\theta_2^2 =  \frac{L^2\eta^2}{2} \zeta(\kappa,R) (d+6)^3 + 2L\delta \zeta(\kappa,R) (d+4)^2 + \frac{2\delta^2}{\eta^2} \zeta(\kappa,R)  d,
\end{align*}
and note that, by definition, $\zeta(\kappa,R)\geq 1$ for all $\kappa$ and $R$. We also note that $(d+4)^2 > d^{1/2}(d+3)^{3/2}$ for all $d\geq 0$.	Hence,
\begin{align*}
	\theta_2^2 &\geq \frac{L^2\eta^2}{2} (d+6)^3 + 2L\delta (d+4)^2 + \frac{2\delta^2}{\eta^2}  d \\
	&>  \frac{L^2\eta^2}{2} (d+3)^3 + 2L\delta d^{1/2}(d+3)^{3/2} + \frac{2\delta^2}{\eta^2}  d \\
	&\stackrel{\text{\eqref{eq:theta_1}}}{>} \theta_1^2,
\end{align*}
which implies that $D = \alpha\max\{\theta_1,\theta_2\} = \alpha \theta_2$, and thus $\Delta = \frac{\alpha \theta_2 + 2V}{1 - \rho}$. 

Now that we have an expression for $\Delta$, the second part of the proof consists in computing $\alpha$ and $\eta$ that minimise $\Delta$. We note that $\Delta$ depends on $\eta$ through $\theta_2$ only, and it depends on $\alpha$ through $\rho$ only. We can see that 
\begin{align*}
	\frac{\partial \theta_2}{\partial \eta} = 0 \Longrightarrow \zeta(\kappa,R)\left( L^2(d+6)^3\bar{\eta} - \frac{4\delta^2 d}{\bar{\eta}^3} \right) = 0,
\end{align*}
from which we conclude that $\bar{\eta} = \left( 4\delta^2d/(L^2(d+6)^3) \right)^{1/4}$ minimises $\theta_2$, and since the denominator of $\Delta$ is independent of $\eta$, $\Delta$ attains its minimum at $\bar{\eta}$.

Next, we obtain $\bar{\alpha}$ that minimises $\Delta$. To that end, we compute the derivative of $\Delta$ evaluated at $\bar{\eta}$ with respect to $\alpha$ and set it to be equal to zero. Note that
\begin{align*}
	\frac{\partial \Delta|_{\eta = \bar{\eta}}}{\partial \alpha} = \frac{\partial}{\partial \alpha}\left(  \frac{\alpha \bar{\theta} + 2V}{1- \sqrt{2(d+4)L^2\zeta(\kappa,R)\alpha^2 - \sigma \alpha + 1}} \right) = 0
\end{align*}
implies
\begin{align*}
	2\bar{\theta} \sqrt{2 L^2 (d + 4) \zeta(\kappa,R)\alpha^2  - \sigma \alpha + 1}  + 8L^2V\zeta(\kappa,R)(d+4)\alpha + \sigma \bar{\theta} \alpha - 2 \sigma V - 2 \bar{\theta} = 0,
\end{align*}
and then
\begin{align}\label{eq:ABC}
	\left( (8L^2V\zeta(\kappa,R)(d+4) + \sigma \bar{\theta}) \alpha - (2 \sigma V + 2 \bar{\theta}) \right)^2 = 4\bar{\theta}^2 \left( 2 L^2 (d + 4) \zeta(\kappa,R)\alpha^2  - \sigma \alpha + 1 \right).
\end{align}
Lastly, we simply group terms in \eqref{eq:ABC} to write $A\alpha^2+B\alpha+C=0$ with $A,B$ and $C$ as per the theorem statement, completing the proof.

\section{Proof of Theorem \ref{theo:regret}.}\label{sec:appendix-regret-1}
To prove Theorem \ref{theo:regret} we need the following two intermediate lemmas.
\begin{lemma}\label{lem:regret-tracking}
	Let $\rho_k \coloneqq \sqrt{2(d+4)L^2\zeta(\kappa,R)\alpha_k^2 - \sigma \alpha_k + 1}$, and suppose $\alpha_k\in (0,\sigma/(2L^2(d+4)\zeta(\kappa,R)))$ for all $k\in\N_0$. Then,
	\begin{align*}
		\sum_{k=0}^{T} \espp{e_k} \leq \frac{1}{1-\max\{\rho_0,\rho_T\}} \left( \espp{e_0} - \rho_T \espp{e_T} + \sum_{k=0}^{T-1} D_{k} + V_T \right),
	\end{align*}
	where $V_T\coloneqq \sum_{k=0}^{T-1} \mathrm{dist}(x_{k^+}^\star,x_{k^++1}^\star)$, and $D_k\coloneqq \alpha_k \sqrt{\left( \frac{L^2\eta_k^2}{2} (d+6)^3 + 2L\delta (d+4)^2 + \frac{2\delta^2}{\eta_k^2} d \right) \zeta}$.
\end{lemma} 
\begin{proof} 
	We know that $e_{k+1}\leq \bar{e}_k + \mathrm{dist}(x_{k^+}^\star,x_{k^++1}^\star)$, and thus
	$\espc{e_{k+1}}{x_k}\leq \espc{\bar{e}_k}{x_k} + \mathrm{dist}(x_{k^+}^\star,x_{k^++1}^\star)$. Note that this is a relaxed version of \eqref{eq:2V} since we do not use the change on minimiser bound from Assumption \ref{assu:standing}(b). Therefore, from the proof of Corollary \ref{coro:delta}, it is easy to see that for any $\alpha_k>0$,
	\begin{align}\label{eq:error-reg}
		\espp{e_{k+1}} \leq \rho_k \espp{e_k} + D_k +  \mathrm{dist}(x_{k^+}^\star,x_{k^++1}^\star).
	\end{align}
	Summing both sides of \eqref{eq:error-reg} over and adding $\espp{e_0}$ to both sides,
	\begin{align*}
		\sum_{k=0}^{T} \espp{e_k} &\leq \espp{e_0} + \sum_{k=1}^{T}\rho_{k-1} \espp{e_{k-1}} + \sum_{k=1}^{T} D_{k-1} + \sum_{k=1}^{T} \mathrm{dist}(x_{k^+-1}^\star,x_{k^+}^\star) \\
		&= \espp{e_0} + \sum_{k=0}^{T-1} \rho_k \espp{e_k} + \sum_{k=0}^{T-1} D_{k} + \sum_{k=0}^{T-1} \mathrm{dist}(x_{k^+}^\star,x_{k^++1}^\star) \\
		&= \espp{e_0} - \rho_T \espp{e_T} +  \sum_{k=0}^{T} \rho_k \espp{e_k} + \sum_{k=0}^{T-1} D_{k} + \sum_{k=0}^{T-1} \mathrm{dist}(x_{k^+}^\star,x_{k^++1}^\star) \\
		&\hspace{-5mm}\leq  \espp{e_0} - \rho_T \espp{e_T} +  \max\{\rho_{0},\rho_T\} \sum_{k=0}^{T}  \espp{e_k} + \sum_{k=0}^{T-1} D_{k} + \sum_{k=0}^{T-1} \mathrm{dist}(x_{k^+}^\star,x_{k^++1}^\star) ,
		%
	\end{align*}
	where the last inequality follows from the definition of $\rho_k$, and the proof is complete by noting that $\max\{\rho_0,\rho_T\}<1$ since $\alpha_k\in (0,\sigma/(2L^2(d+4)\zeta(\kappa,R)))$. 
\end{proof} 

\begin{lemma}\label{lem:regret-estimation}
	Let $\rho_k \coloneqq \sqrt{2(d+4)L^2\zeta(\kappa,R)\alpha_k^2 - \sigma \alpha_k + 1}$, and suppose $\alpha_k\in (0,\sigma/(2L^2(d+4)\zeta(\kappa,R)))$ for all $k\in\N_0$. Then,
	\begin{align*}
		\sum_{k=0}^{T} \espp{\bar{e}_k} \leq \frac{1}{1-\max\{\rho_1,\rho_{T+1}\}}\left(  \espp{\bar{e}_0} - \rho_{T+1} \espp{\bar{e}_T} + \sum_{k=1}^{T} D_k + \max\{\rho_1,\rho_{T+1}\} V_T\right),
	\end{align*}
	where $V_T\coloneqq \sum_{k=0}^{T-1} \mathrm{dist}(x_{k^+}^\star,x_{k^++1}^\star)$.
\end{lemma} 
\begin{proof}	
	From the proof of Corollary \ref{coro:delta}, we can conclude that $\espp{\bar{e}_k} \leq \rho_k \espp{e_k} + D_k$. Then, by the triangle inequality of the Riemannian distance, we can write $e_k = \mathrm{dist}(x_k,x_{k^+}^\star) \leq \mathrm{dist}(x_k,x_{k^+-1}^\star) + \mathrm{dist}(x_{k^+-1}^\star,x_{k^+}^\star)$. Therefore,
	\begin{align*}
		\espp{\bar{e}_k} \leq \rho_k \espp{ \bar{e}_{k-1} } + \rho_k\, \mathrm{dist}(x_{k^+-1}^\star,x_{k^+}^\star)  + D_k.
	\end{align*}
	We proceed similarly to the proof of Lemma \ref{lem:regret-tracking}, that is,
	\begin{align*}
		\sum_{k=0}^{T} \espp{\bar{e}_k} &\leq \espp{\bar{e}_0} + \sum_{k=1}^{T} \rho_k \espp{\bar{e}_{k-1}} +  \sum_{k=1}^{T} \rho_k\, \mathrm{dist}(x_{k^+-1}^\star,x_{k^+}^\star)  + \sum_{k=1}^{T} D_k \\
		&\leq \espp{\bar{e}_0} + \sum_{k=0}^{T-1} \rho_{k+1} \espp{\bar{e}_{k}} +  \sum_{k=0}^{T-1} \rho_{k+1}\, \mathrm{dist}(x_{k^+}^\star,x_{k^+ +1}^\star)  + \sum_{k=1}^{T} D_k \\
		&= \espp{\bar{e}_0} - \rho_{T+1} \espp{\bar{e}_T} + \sum_{k=0}^{T} \rho_{k+1} \espp{\bar{e}_{k}}  +  \sum_{k=0}^{T-1} \rho_{k+1} \mathrm{dist}(x_{k^+}^\star,x_{k^++1}^\star)  + \sum_{k=1}^{T} D_k \\
		&\leq \espp{\bar{e}_0} - \rho_{T+1} \espp{\bar{e}_T} + \max\{\rho_1,\rho_{T+1}\}\sum_{k=0}^{T}  \espp{\bar{e}_{k}}\\
		&\qquad  +  \max\{\rho_1,\rho_{T+1}\}\sum_{k=0}^{T-1}  \mathrm{dist}(x_{k^+}^\star,x_{k^++1}^\star)  + \sum_{k=1}^{T} D_k,
	\end{align*}
	completing the proof.
\end{proof}

Now we can proceed with the proof of Theorem \ref{theo:regret}. Note that the geodesic strong convexity of $f_{k^+}$, the Cauchy-Schwarz inequality, and Assumption \ref{assu:gradient} imply
\begin{align*}
	f_{k^+}(x_k) - f_{k^+}(x_{k^+}^\star) &\leq -\inner{\grad f_{k^+}(x_k),\Exp_{x_k}^{-1}(x_{k^+}^\star)} \\
	&\leq \norm{\grad f_{k^+}(x_k)} e_k \\
	&\leq G e_k \\
	\espp{f_{k^+}(x_k)} - f_{k^+}(x_{k^+}^\star) &\leq G\espp{e_k}.
\end{align*}
Similarly, it is immediate to show that $\espp{f_{k^+}(x_{k+1})} - f_{k^+}(x_{k^+}^\star)\leq G \espp{\bar{e}_k}$. Therefore, by means of Lemmas \ref{lem:regret-tracking}, \ref{lem:regret-estimation}, and the regret definitions in \eqref{eq:regret-def}, we can immediately obtain the following upper bounds on $\reg_T^{\textrm{Track.}}$ and $\reg_T^{\textrm{Est.}}$,
\begin{align*}
	\reg_T^{\textrm{Track.}} &\leq \frac{G}{1-\max\{\rho_0,\rho_T\}} \left( \espp{e_0} - \rho_T \espp{e_T} + \sum_{k=0}^{T-1} D_{k} + V_T \right),\\
	\reg_T^{\textrm{Est.}} &\leq \frac{G}{1-\max\{\rho_1,\rho_{T+1}\}}\left(  \espp{\bar{e}_0} - \rho_{T+1} \espp{\bar{e}_T} + \sum_{k=1}^{T} D_k + \max\{\rho_1,\rho_{T+1}\} V_T\right).
\end{align*}

The next step of the proof follows by using the so-called \emph{doubling-trick} from \cite[Section 2.3.1]{shalev2011online}, which divides the algorithm rounds into periods of increasing size, specifically, in periods of $2^m$ rounds, $m\in\N_0$. Formally, this allows us to write, for any $T\geq 1$,
\begin{align*}
	\sum_{k=0}^{T-1}  D_k \leq \sum_{m=0}^{\lceil \log_2(T)\rceil} \sum_{k=2^m-1}^{2^{m+1}-2} D_k \ .
\end{align*}
Then, if for each period of $2^m$ rounds we can find a step size and oracle's precision such that $D_k \leq \bar{c}/\sqrt{2^m}$ for some $\bar{c}>0$, $k\in [2^m-1,2^{m+1}-2]$, $m\in\N_0$, then we would have
\begin{align*}
	\sum_{k=0}^{T-1}  D_k \leq \sum_{m=0}^{\lceil \log_2(T)\rceil} \sum_{k=2^m-1}^{2^{m+1}-2} D_k 
	&\leq \sum_{m=0}^{\lceil \log_2(T)\rceil} \bar{c} \sqrt{2^m} \\
	&= \bar{c}\ \frac{1 - (\sqrt{2})^{\lceil \log_2(T) \rceil+1}}{1-\sqrt{2}} \\
	&\leq \bar{c}\ \frac{1-\sqrt{2T}}{1-\sqrt{2}} \\
	&\leq \frac{\bar{c}\sqrt{2}}{1-\sqrt{2}}\sqrt{T},
\end{align*}
which would prove the regret bound for $\reg_T^{\textrm{Track.}}$ in \eqref{eq:regret-track}. We can also bound $\sum_{k=1}^{T} D_k$ exactly as above to prove $\reg_T^{\textrm{Est.}}$ in \eqref{eq:regret-est}. Therefore, to conclude the proof, we have to show that the choices of $\alpha_k$ and $\eta_k$ in the theorem statement indeed imply $D_k \leq \bar{c}/\sqrt{2^m}$ in each period of length $2^m$, which is what we do below.

Fix a period of length $2^m$, that is, consider $T_k=2^m$ for $k\in[2^m-1,2^{m+1}-2]$, $m\in\N_0$. Recall that $D_k\coloneqq \alpha_k \sqrt{\left( \frac{L^2\eta_k^2}{2} (d+6)^3 + 2L\delta (d+4)^2 + \frac{2\delta^2}{\eta_k^2} d \right) \zeta}$. Then, in each period, $D_k \leq \bar{c}/\sqrt{T_k}$ implies
\begin{align}\label{eq:eta-eqn}
	\underbrace{\frac{L^2(d+6)^3\zeta}{2}}_{\mathcal{A}\geq 0}{\eta_k^4} + \underbrace{ \left(2L\delta(d+4)^2\zeta - \frac{\bar{c}^2}{\alpha_k^2 T_k}\right)}_{\mathcal{B}_k} {\eta_k^2} + \underbrace{ 2\delta^2d\zeta}_{\mathcal{C} > 0} \leq 0,
\end{align}
which is a quadratic inequality on $\eta_k^2$. The solutions to $\mathcal{A} \eta_k^4 + \mathcal{B}_k \eta_k^2 + \mathcal{C} = 0$ are 
\begin{align*}
	\mathbf{x}_{1} = \frac{-\mathcal{B}_k + (\mathcal{B}_k^2 - 4 \mathcal{A} \mathcal{C})^{\frac{1}{2}}}{2\mathcal{A}}, \quad 
	\mathbf{x}_{2} = \frac{-\mathcal{B}_k - (\mathcal{B}_k^2 - 4 \mathcal{A} \mathcal{C})^{\frac{1}{2}}}{2\mathcal{A}}. 
\end{align*}
Since $\mathcal{A}\geq 0$ and $\mathcal{C}>0$, a necessary condition such that $\mathbf{x}_{1},\mathbf{x}_{2}>0$ is that $\mathcal{B}_k < 0$, which implies that	
$\alpha_k < \sqrt{ \frac{\bar{c}^2}{2L\delta(d+4)^2\zeta(\kappa,R) T_k} }$.
We also need that $\Delta_k \coloneqq \mathcal{B}_k^2 - 4\mathcal{A} \mathcal{C} \geq 0$, which imposes and extra condition on the step size $\alpha_k$, that is, 
\begin{align*}
	0 \leq	\mathcal{B}_k^2 - 4\mathcal{A} \mathcal{C} &= \left(2L\delta(d+4)^2\zeta - \frac{\bar{c}^2}{\alpha_k^2 T_k}\right)^2 - 4L^2\delta^2d(d+6)^3\zeta^2 \\
	&= 4L^2\delta^2(d+4)^4\zeta^2 - 4L^2\delta^2d(d+6)^3 \zeta^2  - \frac{4L\delta(d+4)^2\zeta \bar{c}^2}{\alpha_k^2 T_k} + \frac{\bar{c}^4}{ \alpha_k^4 T_k^2 }
\end{align*}
which can be written as a quadratic equation on $\alpha^2_k$,
\begin{align}\label{eq:alpha-eqn}
	0\leq \underbrace{ \left[ 4L^2\delta^2(d+4)^4\zeta^2 - 4L^2\delta^2d(d+6)^3 \zeta^2 \right]}_{\bar{\mathcal{A}}}{\alpha_k^4}
	\underbrace{ - \frac{4L\delta(d+4)^2\zeta \bar{c}^2}{T_k} }_{\bar{\mathcal{B}}_k\leq 0}{\alpha_k^2} + \underbrace{ \frac{\bar{c}^4}{ T_k^2 } }_{\bar{\mathcal{C}}_k > 0}.
\end{align}
The solutions to $\bar{\mathcal{A}} \alpha_k^4 + \bar{\mathcal{B}} \alpha_k^2 + \bar{\mathcal{C}} = 0$ are
\begin{align*}
	\mathbf{y}_{1} = \frac{-\bar{\mathcal{B}}_k + (\bar{\mathcal{B}}_k^2 - 4 \bar{\mathcal{A}} \bar{\mathcal{C}}_k)^{\frac{1}{2}}}{2\bar{\mathcal{A}}}, \quad
	\mathbf{y}_{2} = \frac{-\bar{\mathcal{B}}_k - (\bar{\mathcal{B}}_k^2 - 4 \bar{\mathcal{A}} \bar{\mathcal{C}}_k)^{\frac{1}{2}}}{2\bar{\mathcal{A}}}.
\end{align*}
We analyse the determinant $\bar{\Delta}_k \coloneqq \bar{\mathcal{B}}_k^2 - 4\bar{\mathcal{A}} \bar{\mathcal{C}}_k$, that is,
\begin{align*}
	\bar{\Delta}_k &= \frac{16L^2\delta^2(d+4)^4\zeta^2 \bar{c}^4}{T_k^2} - 4\left[ 4L^2\delta^2(d+4)^4\zeta^2 - 4L^2\delta^2d(d+6)^3 \zeta^2 \right]\frac{\bar{c}^4}{ T_k^2 } \\
	&= \frac{ 16L^2\delta^2d(d+6)^3 \zeta^2\bar{c}^4 }{ T_k^2 },
\end{align*}
which is always non-negative. Additionally, note that $\bar{\mathcal{A}} \geq 0$ for $d < 4$, and $\bar{\mathcal{A}} \leq 0$ for $d \geq 5$. We will use this fact to write a closed-form choice for the step size $\alpha_k$ in each period $T_k$ of length $2^m$. 

\begin{itemize}
	\item 	For $d < 4$ ($\bar{\mathcal{A}} \geq 0$), we note that $\bar{\mathcal{B}}_k^2 - 4\bar{\mathcal{A}} \bar{\mathcal{C}}_k \leq \bar{\mathcal{B}}^2_k$ since $\bar{\mathcal{C}}_k>0$. This fact together with $\bar{\mathcal{B}}_k \leq 0$ implies that $\mathbf{y}_{1}\geq \mathbf{y}_{2} \geq 0$. Then, from \eqref{eq:alpha-eqn} $\alpha_k$ needs to satisfy $\bar{\mathcal{A}}(\alpha_k^2 - \mathbf{y}_{1})(\alpha_k^2 - \mathbf{y}_{2})\geq 0$,
	which holds for $\alpha_k \geq \sqrt{ \mathbf{y}_{1} }$ or $0<\alpha_k \leq \sqrt{ \mathbf{y}_{2} }$.
	\item For $d \geq 5$ ($\bar{\mathcal{A}} \leq 0$), we note that $\bar{\mathcal{B}}_k^2 - 4\bar{\mathcal{A}} \bar{\mathcal{C}}_k \geq \bar{\mathcal{B}}_k^2$ since $\bar{\mathcal{C}}_k>0$. Moreover, since $\bar{\mathcal{B}}_k \leq 0$, we have that $\mathbf{y}_{1}\leq 0$ and $\mathbf{y}_{2} \geq 0$. Therefore, from \eqref{eq:alpha-eqn}, $\bar{\mathcal{A}}(\alpha_k^2 - \mathbf{y}_{2})(\alpha_k^2 + |\mathbf{y}_{1}|)\geq 0$ which satisfied with $0<\alpha_k \leq \sqrt{\mathbf{y}_{2}}$.
\end{itemize}
In addition, we need that $\alpha_k < \frac{\sigma}{2L^2(d+4)\zeta(\kappa,R)}$ (to make $\rho_k <1$). Therefore, we choose
\begin{align*}
	\alpha_k < \min\left\{\sqrt{\mathbf{y}_{2}} , \frac{\sigma}{2L^2(d+4)\zeta(\kappa,R)} \right\}.
\end{align*}

Since this choice of $\alpha_k$ implies that $\Delta_k \geq 0$, then $\mathcal{A} \eta_k^4 + \mathcal{B}_k \eta_k^2 + \mathcal{C} = 0$ has two distinct positive real solutions $\mathbf{x}_{1}$ and $\mathbf{x}_{2}$. Then, from \eqref{eq:eta-eqn}, we recall that $\eta_k$ needs to be chosen such that $\mathcal{A} \eta_k^4 + \mathcal{B}_k \eta_k^2 + \mathcal{C} \leq 0 \Longleftrightarrow
\mathcal{A}(\eta_k^2 - \mathbf{x}_{1})(\eta_k^2 - \mathbf{x}_{2}) \leq 0$.
Since $\mathcal{A}\geq 0$ and $\mathcal{C} >0$, then $\mathcal{B}_k^2 - 4 \mathcal{A} \mathcal{C} \leq \mathcal{B}_k^2$ and $0\leq \mathbf{x}_{2}\leq \mathbf{x}_{1}$. Therefore, the choice of $\eta_k$ needs to satisfy
\begin{align*}
	\sqrt{ \frac{-\mathcal{B}_k - (\mathcal{B}_k^2 - 4 \mathcal{A} \mathcal{C})^{\frac{1}{2}}}{2\mathcal{A}} } \leq  \eta_k \leq \sqrt{ \frac{-\mathcal{B}_k + (\mathcal{B}_k^2 - 4 \mathcal{A} \mathcal{C})^{\frac{1}{2}}}{2\mathcal{A}} },
\end{align*}
concluding the proof.

\bibliography{bibliography}
\end{document}